\documentclass[12pt]{amsart}
\usepackage{amsmath}
\usepackage{amsfonts}
\usepackage{amssymb}
\usepackage{latexsym}
\usepackage{amsthm}
\usepackage{a4wide}

\newtheorem{theorem}{Theorem}
\newtheorem{lemma}[theorem]{Lemma}

\newtheorem{coro}[theorem]{Corollary}

\newtheorem{proposition}[theorem]{Proposition}

\newcounter{other}            % Questions get letters
\newtheorem{otherth}[other]{Theorem}              % Other papers' theorems
\newtheorem{otherp}[other]{Proposition}% Other papers'propositions
\newtheorem{otherl}[other]{Lemma}        % Other papers' lemmas

\newcommand{\Cn}{\mathbb{C}^n}

\newcommand{\Sn}{\mathbb{S}_ n}

\newcommand{\Bn}{\mathbb{B}_ n}
\newcommand{\D}{\Bbb D}

\numberwithin{equation}{section}
\begin{document}

\title{A Toeplitz type operator on Hardy spaces in the unit ball}
\author[Jordi Pau]{Jordi Pau}
\address{Jordi Pau \\Departament de Matem\`{a}tiques i Inform\`{a}tica \\
Universitat de Barcelona\\
08007 Barcelona\\
Catalonia, Spain} \email{jordi.pau@ub.edu}

\author[Antti Per\"{a}l\"{a}]{Antti Per\"{a}l\"{a}}
\address{Antti Per\"{a}l\"{a} \\Departament de Matem\`{a}tiques i Inform\`{a}tica \\
Universitat de Barcelona\\
08007 Barcelona\\
Catalonia, Spain\\
Barcelona Graduate School of Mathematics (BGSMath).} \email{perala@ub.edu}
%\date{}

%\urladdr{}

%


\thanks{The first author was partially
 supported by  DGICYT grant MTM2014-51834-P
(MCyT/MEC) and the grant 2017SGR358 (Generalitat de Catalunya). The second author acknowledges financial support from the Spanish Ministry of Economy and Competitiveness, through the Mar\'ia de Maeztu Programme for Units of Excellence in R\&D (MDM-2014-0445). Both authors are also supported by the grant MTM2017-83499-P (Ministerio de Educaci\'{o}n y Ciencia).}

%%% ----------------------------------------------------------------------

\begin{abstract}
We study a Toeplitz type operator $Q_\mu$ between the holomorphic Hardy spaces $H^p$ and $H^q$ of the unit ball. Here the generating symbol $\mu$ is assumed to a positive Borel measure. This kind of operator is related to many classical mappings acting on Hardy spaces, such as composition operators, the Volterra type integration operators and Carleson embeddings. We completely characterize the boundedness and compactness of $Q_\mu:H^p\to H^q$ for the full range $1<p,q<\infty$; and also describe the membership in the Schatten classes of $H^2$. In the last section of the paper, we demonstrate the usefulness of $Q_\mu$ through applications.
\end{abstract}

\maketitle

%%% ----------------------------------------------------------------------
%%% ----------------------------------------------------------------------

\section{Introduction and main results}

Let $\Bn=\{z\in \mathbb{C}^n:|z|<1\}$ be the open unit ball in $\Cn$, the Euclidian space of complex dimension $n$.   For any two points $z=(z_ 1,\dots,z_ n)$ and $w=(w_ 1,\dots,w_ n)$ in $\Cn$ we write
\[
\langle z,w\rangle =z_ 1\bar{w}_ 1+\dots +z_ n \bar{w}_ n,
\]
 and
\[
 |z|=\sqrt{\langle z,z\rangle}=\sqrt{|z_ 1|^2+\dots +|z_ n|^2}.
\]
 For a positive Borel measure $\mu$ on $\Bn$,
the Toeplitz type operator $Q_ {\mu}$ is defined as
\begin{displaymath}
Q_{\mu} f(z)=\int_{\Bn} \!\!\frac{f(w)}{(1-\langle z,w \rangle )^n}\,d\mu(w),\qquad z\in \Bn.
\end{displaymath}

In the one variable setting, the operator $Q_{\mu}$ appeared in \cite{Lu}, where a description of the membership of $Q_{\mu}$ in the Schatten ideals $S_ p$ of the Hardy space $H^2$ was obtained. As mentioned in that paper, this operator is closely related with the study of composition operators, and later on in \cite{AS1} a connection with a Volterra type integration operator was given.
As far as we know, it seems that the operator $Q_{\mu}$ has not been studied in the setting of Hardy spaces in the unit ball.

 For $0<p<\infty$, the Hardy space $H^p:=H^p(\Bn)$ consists of those holomorphic functions $f$ in $\Bn$ with
 \[ \|f\|_{H^p}^p=\sup_{0<r<1}\int_{\Sn} \!\! |f(r\zeta)|^p \,d\sigma(\zeta)<\infty,\]
where $d\sigma$ is the surface measure on the unit sphere $\Sn:=\partial \Bn$ normalized so that $\sigma(\Sn)=1$. Moreover, any function in $H^p$ has radial limits $f(\zeta)=\lim_{r\to 1^{-}} f(r\zeta)$ for a.e. $\zeta \in \Sn$; and $H^2$ becomes a Hilbert space when endowed with the inner product

$$\langle f,g\rangle_{H^2}= \int_{\Sn} f(\zeta)\overline{g(\zeta)}d\sigma(\zeta).$$

We refer the reader to the books \cite{Rud} and \cite{ZhuBn} for the theory of Hardy spaces in the unit ball.

We completely describe the boundedness of $Q_{\mu}:H^p\rightarrow H^q$ for $1<p,q<\infty$ (the case $p\neq q$ seems to be new even in one dimension), as well as characterizing its membership in $S_ p(H^2)$, thus generalizing Luecking's results to higher dimensions.

Before stating the main results of the paper, we recall the concept of a Carleson measure.
For $\zeta \in \Sn$ and $\delta>0$ consider the non-isotropic metric balls
\[B_{\delta}(\zeta)=\big \{ z\in \Bn: |1-\langle z, \zeta \rangle |<\delta \big \}.\]
A positive Borel measure $\mu$ on $\Bn$ is said to be a Carleson measure if there exists a constant $C>0$ such that
\[\mu \big (B_{\delta}(\zeta)\big )\le C \delta \, ^ n\]
for all $\zeta \in \Sn$ and $\delta>0$. Obviously every Carleson measure is finite. H\"{o}rmander \cite{H} extended to several complex variables the famous Carleson measure embedding theorem \cite{Car0, Car1} asserting that, for $0<p<\infty$, the embedding $I_ d:H^p\rightarrow L^p(\mu):=L^p(\Bn,d\mu)$ is bounded if and only if $\mu$ is a Carleson measure.

More generally, for $s>0$, a finite positive Borel measure on $\Bn$ is called an $s$-Carleson measure if there exists a constant $C>0$ such that $\mu(B_{\delta}(\zeta))\le C \delta \,^{ns}$ for all $\zeta\in \Sn$ and $\delta>0$. We denote by $\|\mu\|_{CM_s}$ the infimum of all possible $C$ above.

 It is well known (see \cite[Theorem 45]{ZZ}) that $\mu$ is an $s$-Carleson measure if and only if for each (some) $t>0$ one has
\begin{equation}\label{sCM}
\sup_{a\in \Bn}\int_{\Bn} \!\!\frac{(1-|a|^2)^t}{|1-\langle a,z \rangle |^{ns+t}} \,d\mu(z)<\infty.
\end{equation}
Moreover, with constants depending on $t$, the supremum of the above integral is comparable to $\|\mu\|_{CM_s}$.

In \cite{Du}, Duren gave an extension of Carleson's theorem by showing that, for
 $0<p<q<\infty$, one has that $I_ d:H^p\rightarrow L^q(\mu)$ is bounded if and only if $\mu$ is a $q/p$-Carleson measure. Moreover, one has the estimate $\|I_ d\|_{H^p\rightarrow L^q(\mu)}\asymp \|\mu\|^{1/q}_{CM_{q/p}}$. A simple proof of this result, in the setting of the unit ball, can be found in \cite{P1} for example.\\

Our first result characterizes the boundedness of the Toeplitz type operator $Q_{\mu}:H^p\rightarrow H^q$ when $1<p\le q$ in terms of $s$-Carleson measures.
\begin{theorem}\label{T2}
Let $1<p\le q<\infty$ and $\mu$ be a positive Borel measure on $\Bn$. Then $Q_{\mu}:H^p\rightarrow H^q$ is bounded  if and only if $\mu$ is an $(1+\frac{1}{p}-\frac{1}{q})$- Carleson measure. Moreover,
\[
\big \| Q_{\mu} \big \|_{H^p\rightarrow H^q} \asymp \big \|\mu \big \|_{CM_ s}.
\]
\end{theorem}
\mbox{}
\\
The notation $A\asymp B$ means that the two quantities are comparable, and $\|T\|_{X\rightarrow Y}$ denotes the norm of the operator $T:X\rightarrow Y$.\\

For $\zeta \in \Sn$ and $\gamma>1$ the admissible approach region $\Gamma_{\gamma}(\zeta)$ is defined as
\begin{displaymath}
\Gamma(\zeta)=\Gamma_{\gamma}(\zeta)=\left \{z\in \Bn: |1-\langle z,\zeta\rangle |<\frac{\gamma}{2} (1-|z|^2) \right \}.
\end{displaymath}
As we will see later, for most of the properties we will use, the choice of $\gamma$ is not important.\\

For a positive Borel measure $\mu$ on $\Bn$, we set
\[
 \widetilde{\mu}(\zeta)=\int_{\Gamma(\zeta)} (1-|z|^2)^{-n}d\mu(z),\qquad \zeta\in \Sn.
\]
The characterization of the boundedness of $Q_{\mu}$ from $H^p$ to $H^q$ in the case $1<q<p<\infty$ will be given in terms of the function $\widetilde{\mu}$.
\begin{theorem}\label{T3}
Let $1<q<p<\infty$ and $\mu$ be a positive Borel measure on $\Bn$. Then $Q_{\mu}:H^p\rightarrow H^q$ is bounded  if and only if $\widetilde{\mu}$ belongs to $L^{r}(\Sn)$ with $r=pq/(p-q)$. Moreover, one has
\[
\big \| Q_{\mu} \big \|_{H^p\rightarrow H^q} \asymp \big \|\widetilde{\mu} \big \|_{L^r(\Sn)}.
\]
\end{theorem}

For $0<p<\infty$,  a compact operator $T$ acting on a separable Hilbert space $H$ belongs to the Schatten class $S_ p:=S_ p(H)$ if its sequence of singular numbers belongs to the sequence space $\ell^p$ (the singular numbers are the square roots of the eigenvalues of the positive operator $T^*T$, where $T^*$ is the Hilbert adjoint of $T$). We refer to \cite[Chapter 1]{Zhu} for a brief account on Schatten classes.

Our next main result (see Theorem \ref{t-SC} in section \ref{SSC}) is a complete characterization of the membership in the Schatten class $S_ p(H^2)$ of the Toeplitz type operator $Q_{\mu}$. One description is similar to the one obtained by Luecking \cite{Lue2} in the one dimensional case, and we also obtain another description in terms of a Berezin type transform.

The paper is organized as follows: first some background and preliminary results are given in section \ref{s2}. In section \ref{S3} we prove Theorem \ref{T2}, and section \ref{s4} is devoted to the proof of Theorem \ref{T3}. A description of the compactness of $Q_{\mu}:H^p\rightarrow H^q$ for $1<p,q<\infty$ is obtained in section \ref{sComp}, and in section \ref{SSC} a characterization of the membership of $Q_{\mu}$ in the Schatten ideal $S_ p(H^2)$ is provided. Finally, the last section contains some applications to weighted composition operators, Volterra type integration operators and Carleson embeddings.

Finally some words on the notation. For $1<p<\infty$, we let $p'$ to denote the conjugate exponent of $p$. We use $dv$ for the normalized volume measure on $\Bn$, and for $\alpha>-1$, we set $dv_{\alpha}(z)=c_{\alpha}(1-|z|^2)^{\alpha} dv(z)$, where $c_{\alpha}$ is a constant taken so that $v_{\alpha}(\Bn)=1$. The notation $a\lesssim b$ means that there is a finite positive constant $C$ with $a\leq C b$. Also, we use the notation $a\gtrsim b$ to indicate that $b\lesssim a$.

\section{Preliminaries}\label{s2}
In this section we collect some facts needed for the proofs of the main results.

\subsection{Admissible maximal and area functions}

For $\zeta \in \Sn$ and $\gamma>1$, recall that the admissible approach region $\Gamma_{\gamma}(\zeta)$ is defined as
\begin{displaymath}
\Gamma(\zeta)=\Gamma_{\gamma}(\zeta)=\left \{z\in \Bn: |1-\langle z,\zeta\rangle |<\frac{\gamma}{2} (1-|z|^2) \right \}.
\end{displaymath}
If $I(z)=\{\zeta \in \Sn: z\in \Gamma(\zeta)\}$, then $\sigma(I(z))\asymp (1-|z|^2)^{n}$, and it follows from Fubini's theorem that, for a positive function $\varphi$, and a finite positive measure $\nu$, one has
\begin{equation}\label{EqG}
\int_{\Bn} \varphi(z)\,d\nu(z)\asymp \int_{\Sn} \left (\int_{\Gamma(\zeta)} \varphi(z) \frac{d\nu(z)}{(1-|z|^2)^{n}} \right )d\sigma(\zeta).
\end{equation}
This fact will be used repeatedly throughout the paper. \\

For $\gamma>1$ and $f$ continuous on $\Bn$, the admissible maximal function $f^*_{\gamma}$ is defined on $\Sn$ by
\[ f^*(\zeta)=f_{\gamma}^*(\zeta)=\sup_{z\in \Gamma_{\gamma}(\zeta)} |f(z)|.\]
We need the following well known result on the $L^p$-boundedness of the admissible maximal function that can be found in \cite[Theorem 5.6.5]{Rud} or \cite[Theorem 4.24]{ZhuBn}.
\begin{otherth}\label{NTMT}
Let $0<p<\infty$ and $f\in H(\Bn)$. Then
\[ \|f^*\|_{L^p(\Sn)}\le C \|f\|_{H^p}.\]
\end{otherth}
Another function we need is the admissible area function $A_{\gamma}f$ defined on $\Sn$ by
\[Af(\zeta)=A_{\gamma}f(\zeta)=\left ( \int_{\Gamma_{\gamma}(\zeta)} |Rf(z)|^2 \,(1-|z|^2)^{1-n}dv(z)\right )^{1/2}.\]
For a function $f$ holomorphic in $\Bn$, here $Rf$ denotes the radial derivative of $f$, that is,
\[Rf(z)= \sum_{k=1}^{n} z_ k \frac{\partial f}{\partial z_ k} (z),\qquad z=(z_ 1,\dots,z_ n)\in \Bn.\]

The following result \cite{AB, FS, P1, Pav1} characterizes the membership in the Hardy space in terms of the admissible area function.
\begin{otherth}\label{AreaT}
Let $0<p<\infty$ and $g\in H(\Bn)$. Then $g\in H^p$ if and only if $Ag\in L^p(\Sn)$. Moreover, if $g(0)=0$ then
\[ \|g\|_{H^p}\asymp \|Ag\|_{L^p(\Sn)}.\]
\end{otherth}
As said before, all the results here are independent of the aperture $\gamma>1$ and, for that reason, from now on we omit it from the notation.\\
\mbox{}
\\
\textbf{Luecking's theorem:}
We will also need the following result essentially due to Luecking \cite{Lue1} (see also \cite{P1}) describing those positive Borel measures for which the embedding from $H^p$ into $L^s(\mu)$ is bounded when $s<p$.
\begin{otherth}\label{LT}
Let $0<s<p<\infty$ and let $\mu$ be a positive Borel measure on $\Bn$. Then the identity $I_ d:H^p\rightarrow L^s(\mu)$ is bounded, if and only if, the function defined on $\Sn$ by
\[ \widetilde{\mu}(\zeta)=\int_{\Gamma(\zeta)} (1-|z|^2)^{-n}d\mu(z) \]
belongs to $L^{p/(p-s)}(\Sn)$. Moreover, one has $\|I_d\|_{H^p\rightarrow L^s(\mu)}\asymp \|\widetilde{\mu}\|_{L^{p/(p-s)}(\Sn)}^{1/s}.$
\end{otherth}
\mbox{}
\\
Finally, we will use the following integral estimate. It can be found in \cite{Ars} and \cite{Jev}.

\begin{otherl}\label{Gamma}
Let $0<s<\infty$ and $\lambda>n\max(1,1/s)$. If $\mu$ is a positive measure, then
$$\int_{\Sn}\left[\int_{\Bn}\left(\frac{1-|z|^2}{|1-\langle z,\zeta\rangle|}\right)d\mu(z)\right]^s d\sigma(\zeta)\leq C \int_{\Sn}\mu(\Gamma(\zeta))^s d\sigma(\zeta).$$
\end{otherl}
\mbox{}
\\
\subsection{Separated sequences and lattices}

A sequence of points $\{z_ j\}\subset \Bn$ is said to be separated if there exists $\delta>0$ such that $\beta(z_ i,z_ j)\ge \delta$ for all $i$ and $j$ with $i\neq j$, where $\beta(z,w)$ denotes the Bergman metric on $\Bn$. This implies that there is $r>0$ such that the Bergman metric balls $D_ j=\{z\in \Bn :\beta(z,z_ j)<r\}$ are pairwise disjoint. \\

Let $D(a,r)=\{z\in \Bn : \beta(a,z)<r\}$ be the Bergman metric ball of radius $r>0$ centered at a point $a\in \Bn$. We need a well-known
result on decomposition of the unit ball $\Bn$.
By Theorem 2.23 in \cite{ZhuBn},
there exists a positive integer $N$ such that for any $0<r<1$ we can
find a sequence $\{a_k\}$ in $\Bn$ with the following properties:
\begin{itemize}
\item[(i)] $\Bn=\cup_{k}D(a_k,r)$.
\item[(ii)] The sets $D(a_k,r/4)$ are mutually disjoint.
\item[(iii)] Each point $z\in\Bn$ belongs to at most $N$ of the sets $D(a_k,4r)$.
\end{itemize}

Any sequence $\{a_k\}$ satisfying the above conditions  is called
an $r$-\emph{lattice}
(in the Bergman metric). Obviously any $r$-lattice is a separated sequence.\\

%Finally, a given a set $M\subset \Bn$, we say that the sequence $\{a_k\}$ is an $r$-lattice over $M$, if it satisfies (ii) and (iii) above, and
%$$M\subset \bigcup_{k}D(a_k,r).$$

\subsection{Tent spaces of sequences}

For $0<p,q<\infty$ and a fixed separated sequence $Z=\{z_ j\}\subset \Bn$, let the tent space $T^p_ q=T^p_ q(Z)$ consist of those sequences $\lambda=\{\lambda_ j\}$ of complex numbers with
\[ \|\lambda \|_{T^p_ q}^p =\int_{\Sn} \!\!\Big (\!\!\sum_{z_ j\in \Gamma(\zeta)} \!|\lambda_ j |^q \Big )^{p/q} d\sigma(\zeta)  <\infty.\]

\begin{otherp}\label{TKL}
Let $Z=\{a_ j\}$ be a separated sequence in $\Bn$ and let $0<p<\infty$. If $b>n\max(1,2/p)$, then the operator $T_{Z}: T^p_ 2(Z)\rightarrow H^p$ defined by
$$T_{Z}(\{\lambda_ j\})=\sum_ j \lambda_ j \,\frac{(1-|a_ j|^2)^{b}}{(1-\langle z, a_ j  \rangle )^b}$$ is bounded.
\end{otherp}
\begin{proof}
See for example \cite{Ars, Jev,Lue1} or \cite{P1}.
\end{proof}
The sequence space $T^p_{\infty}=T^p_{\infty}(Z)$ consist of those sequences $\lambda=\{\lambda_ j\}$ of complex numbers with $\sup_{a_ k\in \Gamma(\zeta)} |\lambda_ k|\in L^p(\Sn)$. Set
\[\|\lambda \|_{T^p_ {\infty}}^p =\int_{\Sn}\Big (\sup_{a_ k\in \Gamma(\zeta)} |\lambda_ k| \Big )^p d\sigma(\zeta).\]
\begin{otherth}\label{TTD1}
Let $1<p<\infty$. The dual of $T^p_ 1$ can be identified with $T^{p'}_{\infty}$ under the pairing
\[\langle \lambda ,\mu \rangle =\sum_ k \lambda _ k \,\overline{\mu_ k} (1-|a_ k|^2)^n.\]
Under the same pairing, for $1<q<\infty$, the dual of $T^p_ q$ can be identified with $T^{p'}_{q'}$.
\end{otherth}
\begin{proof}
Again see \cite{Ars, Jev,Lue1}.
\end{proof}

\subsection{Forelli-Rudin type estimates}
We need the following well known integral estimate that has become very useful in this area of analysis (see \cite[Theorem 1.12]{ZhuBn} for example).
\begin{otherl}\label{IctBn}
Let $t>-1$ and $s>0$. There is a positive constant $C$ such that
\[ \int_{\Bn} \frac{(1-|w|^2)^t\,dv(w)}{|1-\langle z,w\rangle |^{n+1+t+s}}\le C\,(1-|z|^2)^{-s}\]
for all $z\in \Bn$.
\end{otherl}

We also need the following well known discrete version of the previous lemma.

\begin{otherl}\label{l2}
Let $\{z_ k\}$ be a separated sequence in $\Bn$, and let $n<t<s$.
Then
$$
\sum_{k}\frac{(1-|z_ k|^2)^t}{|1-\langle z,z_ k \rangle |^s}\le C\,
(1-|z|^2)^{t-s},\qquad z\in \Bn.
$$
\end{otherl}

The following more general version of Lemma \ref{IctBn} will also be needed. The proof can be found in \cite{OF}.

\begin{otherl}\label{FRgeneral}
Let $s>-1$, $s+n+1>r,t>0$, and $r+t-s>n+1$. For $a \in \Bn$ and $z\in \overline{\mathbb{B}}_ n$, one has
$$\int_{\Bn} \frac{(1-|w|^2)^s dv(w)}{|1-\langle z,w\rangle|^r |1-\langle a,w\rangle|^t} \leq C \frac{1}{|1-\langle z,a\rangle|^{r+t-s-n-1}}.$$
\end{otherl}
\mbox{}
\\

\subsection{Differential type operators}
We need to use the differential and integral type operators $R^{\alpha,t}$ and $R_{\alpha,t}$ for $\alpha\ge -1$ and $t\ge 0$ (see \cite[Section 1.4]{ZhuBn}). Recall that $R^{\alpha,t}$  is the unique continuous linear operator on $H(\Bn)$ satisfying
\begin{displaymath}
R^{\alpha,t}\left (\frac{1}{(1-\langle z,w \rangle )^{n+1+\alpha}}\right )=\frac{1}{(1-\langle z,w \rangle )^{n+1+\alpha+t}}
\end{displaymath}
for all $w\in \Bn$. Similarly, $R_{\alpha,t}$  is the unique continuous linear operator on $H(\Bn)$ satisfying
\begin{displaymath}
R_{\alpha,t}\left (\frac{1}{(1-\langle z,w \rangle)^{n+1+\alpha+t}}\right )=\frac{1}{(1-\langle z,w \rangle )^{n+1+\alpha}}
\end{displaymath}
for all $w\in \Bn$. It is well-known that $$R_{\alpha,t}R^{\alpha,t}=R^{\alpha,t}R_{\alpha,t}=I_d.$$
Most of the time we use these operators as follows. If a holomorphic function $f$ in $\Bn$ has an integral representation
\[
f(z)=\int_{\Bn} \frac{d\nu(w)}{(1-\langle z,w \rangle)^{n+1+\alpha}},
\]
then
\[
R^{\alpha,t} f(z)=\int_{\Bn} \frac{d\nu(w)}{(1-\langle z,w \rangle)^{n+1+\alpha+t}}.
\]

%In addition to the operators $R^{\alpha,t}$, we will also use to radial derivative operator $R$, which is defined for holomorphic $f$ as
%$$Rf(z)=\sum_{k=1}^\infty z_k \frac{\partial f}{\partial z_k}=\sum_{k=1}^n kf_k(z),$$
%where $f(z)=\sum_{k=1}^n f_k(z)$ is the homogeneous expansion of $f$.

\subsection{Khinchine and Kahane inequalities}  Consider a sequence of Rademacher functions $r_ k(t)$ (see \cite[Appendix A]{duren1}). For almost every $t\in (0,1)$ the sequence $\{r_ k(t)\}$ consists of signs $\pm 1$. We state first the classical Khinchine's inequality (see \cite[Appendix A]{duren1} for example).\\
\mbox{}
\\
\textbf{Khinchines's inequality:} Let $0<p<\infty$. Then for any sequence $\{c_k\}$ of complex numbers, we have
\begin{displaymath}
\left(\sum_k |c_k|^2\right)^{p/2}\asymp \int_0^1 \left|\sum_k c_kr_k(t)\right|^pdt.
\end{displaymath}
\mbox{}
\\
The next result is known as Kahane's inequality; see for instance Lemma 5 of Luecking \cite{Lue2} or the paper of Kalton \cite{Kal}.\\
\mbox{}
\\
\textbf{Kahane's inequality:} Let $X$ be a Banach space,  and $0<p,q<\infty$. For any sequence $\{x_ k\}\subset X$, one has
\begin{displaymath}
\left ( \int_{0}^1 \Big \|\sum_ k r_ k(t)\, x_ k \Big \|_{X}^q dt\right )^{1/q} \asymp \left ( \int_{0}^1 \Big \|\sum_ k r_ k(t)\, x_ k \Big \|_{X}^p dt\right )^{1/p}.
\end{displaymath}

\section{Proof of Theorem \ref{T2}}\label{S3}

%We observe first that, if $\mu$ is a finite measure, and $f\in H^{\infty}$, then
%\[
%|Q_{\mu} f(z)| \le \frac{\|f\|_{\infty}}{(1-|z|)^n}\,\mu(\Bn)<\infty.
%\]
%Thus $Q_{\mu}$ is well defined on $H^{\infty}$, a dense subset of $H^p$. \\

\subsection{\textbf{Necessity}}
If $Q_{\mu}:H^p\rightarrow H^q$ is bounded, by the pointwise estimate for $H^p$-functions (see \cite[Theorem 4.17]{ZhuBn}) we get
\[
\big |Q_{\mu} f(z) \big  | \le \frac{1}{(1-|z|^2)^{n/q}} \,\big \| Q_{\mu} f \big \|_{H^q}.
\]
Taking the function $f$ to be the reproducing kernel $K_ z$ of $H^2$, that is,
\[
f (w)= K_ z (w)=(1-\langle w,z \rangle )^{-n}
\]
and taking into account that (an immediate application of \cite[Theorem 1.12]{ZhuBn})
\[
\|K_ z\|_{H^p}\lesssim (1-|z|^2)^{-n(p-1)/p},
\]
we obtain
\[
\begin{split}
\int_{\Bn} \frac{d\mu(w)}{|1-\langle z,w \rangle |^{2n}}&=\big |Q_{\mu} K_ z(z) \big  |
\\
&\le \frac{1}{(1-|z|^2)^{n/q}} \,\,\big \| Q_{\mu}\|_{H^p\rightarrow H^q} \cdot \|K_ z\|_{H^p}
\\
&\le \frac{C}{(1-|z|^2)^{n \big(1/q+(p-1)/p \big)}} \,\,\big \| Q_{\mu}\|_{H^p\rightarrow H^q}.
\end{split}
\]
Hence, from \eqref{sCM} with $t=n (\frac{1}{q}+\frac{(p-1)}{p} )$ we see that $\mu$ is an $s$-Carleson measure with $s=1+\frac{1}{p}-\frac{1}{q}$, and moreover
\[
\|\mu\|_{CM_ s}\le C \big \| Q_{\mu}\|_{H^p\rightarrow H^q}.
\]

\subsection{\textbf{Sufficiency}}

Let $1<p\le q<\infty$ and let $\mu$ be an $s$-Carleson measure with $s=1+\frac{1}{p}-\frac{1}{q}$. Take $t> 0$ satisfying
$n+t>ns$. By the density of the holomorphic polynomials and duality, it suffices to show that
\begin{equation}\label{E-B1}
I_ t(Q_{\mu}f,g):=\left |\int_{\Bn} R^{-1,t}(Q_{\mu}f)(z)\,\overline{R^{t-1,t}g(z)}\,dv_ {2t-1}(z) \right | \le C \|f\|_{H^p}\cdot \|g\|_{H^{q'}}
\end{equation}
for holomorphic polynomials $f$ and $g$ (it is easy to see that, for holomorphic polynomials $P$ and $Q$, one has  $\langle P,Q \rangle _{H^2}=\int_{\Bn} R^{-1,t} P(z)\overline{R^{t-1,t} Q(z)}\, dv_{2t-1}(z)$). Observe that $R^{t-1,t}g$ is also a holomorphic polynomial (this follows from the expression of $R^{t-1,t}g$ in terms of the homogeneous expansion of $g$. See \cite[Chapter 1] {ZhuBn}). This together with \eqref{sCM}, gives

\[
\begin{split}
\int _{\Bn}\left ( \int_{\Bn} \frac{|f(w)|\,d\mu(w)}{|1-\langle w,z\rangle |^{n+t}}\right ) &|R^{t-1,t}g(z)|\,dv_{2t-1}(z)\\
&\le \|f\|_{\infty}\cdot \|R^{t-1,t} g\|_{\infty} \int_{\Bn} \left (\int_{\Bn}\frac{d\mu(w)}{|1-\langle w,z \rangle |^{n+t}}\right ) dv_{2t-1}(z)\\
\\
& \lesssim \|f\|_{\infty}\cdot \|R^{t-1,t}g\|_{\infty} \cdot\|\mu\|_{CM_ s} \int_{\Bn} dv_{t-1+ns-n}(z)<\infty,
\end{split}
\]
 and we can use Fubini's theorem and the properties of the operators $R^{\beta,t}$ to get
\[
\begin{split}
\int_{\Bn} R^{-1,t}(Q_{\mu}f)(z)\,\overline{R^{t-1,t}g(z)}\,dv_{2t-1}(z)&=\int_{\Bn} \left ( \int_{\Bn} \frac{f(w)\,d\mu(w)}{(1-\langle z,w \rangle )^{n+t}}\right ) \overline{R^{t-1,t}g(z)}\,dv_{2t-1}(z)
\\
&=\int_{\Bn} f(w) \left (\int_{\Bn} \frac{\overline{R^{t-1,t}g(z)}\,dv_{2t-1}(z)}{(1-\langle z,w \rangle )^{n+t}} \right ) \,d\mu(w)
\\
& =\int_{\Bn} f(w)\,\overline{R_{t-1,t}R^{t-1,t}g(w)}d\mu(w)
\\
& =\int_{\Bn} f(w)\,\overline{g(w)} \, d\mu(w).
\end{split}
\]
This and H\"{o}lder's inequality give
\begin{equation}\label{E-B2}
I_ t(Q_{\mu}f,g) \le \int_{\Bn} |f(w)\,g(w)|\,d\mu(w)\le \|f\|_{L^{\sigma}(\mu)}\cdot \|g\|_{L^{\sigma'}(\mu)}
\end{equation}
with $\sigma=ps\ge p$.
As $\mu$ is a $(\sigma/p)$-Carleson measure, by Carleson-Duren's theorem we have
\[
\|f\|_{L^{\sigma}(\mu)} \le C_ 1 \big \|\mu \big \|^{1/\sigma}_{CM_ s}\cdot \|f\|_{H^p}.
\] Also,
\[
\frac{\sigma '}{q'}=\frac{ps(q-1)}{q(ps-1)}=\frac{ps(q-1)}{q(p-p/q)}=s.
\]
Therefore, we also have
\[
\|g\|_{L^{\sigma'}(\mu)}\le C_ 2 \big \|\mu \big \|^{1/\sigma'}_{CM_ s}\cdot \|g\|_{H^{q'}}.
\]
 Bearing in mind \eqref{E-B2}, we see that \eqref{E-B1} holds, proving that $Q_{\mu}:H^{p}\rightarrow H^q$ is bounded, and moreover
\[
\big \|Q_{\mu}\big \|_{H^p\rightarrow H^q} \lesssim \|\mu\|_{CM_ s}.
\]

\section{Proof of Theorem \ref{T3}}\label{s4}

\subsection{\textbf{Sufficiency}}

Suppose first that $\widetilde{\mu}$ belongs to $L^{r}(\Sn)$. Observe that $r>1$ so that by Luecking's theorem it follows that
$$\int_{\Bn} |f(z)|^s \,d\mu(z)\le C \|f\|_{H^p}^s$$
whenever $f$ is in $H^p$, with $s=p+1-\frac{p}{q}$. Testing this inequality on the functions
\[
f_ a(z)=\frac{1}{(1-\langle z,a \rangle )^{\sigma}},\qquad a\in \Bn,
\]
with $\sigma$ big enough, we see that $\mu$ is an $s/p$-Carleson measure. Note that, as $r>1$ and $q<p$, one has $0<s/p<1$. Let $t>0$ with $t+ns/p>n$ (observe that this implies $t+n>ns/p$ because $s/p<1$). As in the previous proof, and with the same notation as in \eqref{E-B1}, we need to show that $$|I_ t(Q_{\mu} f,g) | \le  C\, \|\widetilde{\mu}\|_{L^r(\Sn)}\cdot\|f\|_{H^p}\cdot \|g\|_{H^{q'}}$$
for holomorphic polynomials $f$ and $g$. Because $\mu$ is an $(s/p)$-Carleson measure, proceeding as before we can justify the use of Fubini's theorem that gives
\[
\begin{split}
\big | I_ t(Q_{\mu} f,g) \big | & \le \int_{\Bn} |f(z)| \,|g(z)|\,d\mu(z)\asymp \int_{\Sn} \int_{\Gamma(\zeta)}\frac{|fg(z)|}{(1-|z|^2)^{n}} d\mu(z)\,d\sigma(\zeta)\\
\\
&\le \int_{\Sn} |(fg)^*(\zeta)|\,\widetilde{\mu} (\zeta)\,d\sigma(\zeta) \le C \,\|\widetilde{\mu}\|_{L^r(\Sn)} \cdot \|fg \|_{H^{r'}}
\end{split}
\]
with $r'$ being the conjugate exponent of $r$ (the last inequality follows from H\"{o}lder's inequality and Theorem \ref{NTMT}). Finally,  H\"{o}lder's inequality gives $\|fg\|_{H^{r'}}\le \|f\|_{H^p}\cdot \|g\|_{H^{q'}}$ finishing the proof of the sufficiency.\\

%$$t=\left (\frac{pq}{p-q} \right )'= \frac{pq}{pq-p+q}.$$
%$$ (pq-p+q)/q>1\Leftrightarrow pq-p+q>q\Leftrightarrow pq-p=p(q-1)>0$$
%Per altra banda,
%$$(p/r')'=\frac{p}{p-r'}=\frac{p}{p-\frac{pq}{pq-p+q}}=\frac{pq-p+q}{pq-p+q-q}=\frac{pq-p+q}{p(q-1)}$$
%Therefore
%$$(p/t)' t=\frac{pq-p+q}{p(q-1)}\cdot \frac{pq}{pq-p+q}=\frac{q}{q-1}=q'$$
%Therefore Holder's inequality gives $\|fg\|_{H^{t}}\le \|f\|_{H^p}\cdot \|g\|_{H^{q'}}$ finishing the proof of the sufficiency.\\

\subsection{\textbf{Preliminaries for the necessity}}

Set $Q(0)=\Bn$, and given $w \in \Bn\setminus \{0\}$, we write $\zeta_w =w/|w|$ and denote
$$Q(w)=\big \{ z\in \Bn: |1-\langle z, \zeta_w \rangle |<1-|w| \big \}.$$

\begin{lemma}\label{LN1}
Let $w\in \Bn$. Then $1-|w|\asymp |1-\langle z,w \rangle |$ for $z\in Q(w)$.
\end{lemma}
\begin{proof}
The result is trivial for $w=0$. Hence, assume that $w\neq 0$ and $z\in Q(w)$. Then
\begin{displaymath}
\begin{split}
|1-\langle z,w \rangle | &\lesssim |1-\langle z,\zeta_ w \rangle |+|1-\langle \zeta_ w,w \rangle |\\
\\
& \le (1-|w|)+|1-\langle \zeta_ w,w \rangle |=2(1-|w|).
\end{split}
\end{displaymath}
Since the other inequality is trivial, we are done.
\end{proof}

\begin{lemma}\label{discrete}
Let $1<p<\infty$, $\nu$ be a positive Borel measure, finite on compact sets, and consider the general area operator
$$A_\nu g(\zeta)=\left(\int_{\Gamma(\zeta)}|g(z)|^2 d\nu(z)\right)^{1/2}$$
and the discrete maximal operator
$$C_\nu^*(g)(\zeta)=\left(\sup_{a_k \in \Gamma(\zeta)}\frac{1}{(1-|a_k|)^n}\int_{Q(a_k)}(1-|z|^2)^n |g(z)|^2 d\nu(z)\right)^{1/2}
,$$
where $\{a_k\}$ is an $r$-lattice. Suppose that $g\in L^2(\Bn,d\nu_ n)$ with $d\nu_ n(z)=(1-|z|^2)^n d\nu(z)$. If $r$ is small enough and $C_\nu^*g \in L^p(\Sn)$, then $A_\nu g \in L^p(\Sn)$. Moreover,  $r>0$ can be chosen to guarantee that
$$\|A_\nu g\|_{L^p(\Sn)}\lesssim \|C_\nu^*g\|_{L^p(\Sn)}+\|g\|_{L^2(\nu_ n)}.$$
\end{lemma}

\begin{proof}
The proof uses some terminology and concepts related to tent spaces, which are only covered in this paper for discrete measures. Moreover, a substantial part of the proof is remarkably similar to a known standard proof, see for instance \cite{Pel}. For completeness, we will prove this lemma by using notation analogous to that of the mentioned reference, and the reader should have no difficulties comparing the two proofs.

We want to show
\[
\big |\langle f,g \rangle_{T^2_2(\nu)}\big |\lesssim \int_{\Sn} A_\nu(f)(\zeta)C^*_\nu(g)(\zeta)d\sigma(\zeta)+ \|A_{\nu} f\|_{L^1(\Sn)}\cdot\|g\|_{L^2(\nu_ n)}.
\]
By duality of tent spaces, it then follows that if $C^*_\nu g \in L^p(\Sn)$, then $g$ belongs to the dual of $T^{p'}_2(\nu)$, which is $T^p_2(\nu)$. This implies that  $A_\nu g \in L^p(\Sn)$ with the desired estimate.\\

Fix $K>0$ large enough to be specified later, and set $M_K=\{|z|^2 >1-1/K^2\}$. Split the measure $\nu$ as $\nu=\nu_{|_{M_ K}} +\nu'$ with $\nu'=\nu_{|_{\Bn\setminus M_ K}}$. Then
\[
\langle f,g \rangle_{T^2_2(\nu)}=\langle f,g \rangle_{T^2_2(\nu_{|_{M_ K}})}+\langle f,g \rangle_{T^2_2(\nu')}.
\]
Moreover, we have
\[
\begin{split}
\big |\langle f,g \rangle_{T^2_2(\nu')}\big | & \le \int_{\Bn} |f(z)|\,|g(z)|\,(1-|z|^2)^n\,d\nu'(z)
\\
&\asymp \int_{\Sn} \left (\int_{\Gamma(\zeta)} |f(z)|\,|g(z)|\,d\nu'(z)\right )\,d\sigma(\zeta)
\\
&\le \|g\|_{L^2(\nu_ n)}\int_{\Sn} \left (\int_{\Gamma(\zeta)} |f(z)|^2\,d\nu'(z)\right )^{1/2}\,d\sigma(\zeta)
\\
&\le K^n\,\|A_{\nu} f\|_{L^1(\Sn)}\cdot\|g\|_{L^2(\nu_ n)}.
\end{split}
\]
Therefore, it is enough to show that, for $\nu$ supported on $M_ K$ one has
\[
\big |\langle f,g \rangle_{T^2_2(\nu)}\big |\lesssim \int_{\Sn} A_\nu(f)(\zeta)\,C^*_\nu(g)(\zeta)\,d\sigma(\zeta).
\]

\mbox{}
\\
We note first that if $\gamma>1$, $\zeta \in \Sn$, and $z \in \Gamma(\zeta)$, we have for $R \in (0,1)$ the estimate
\begin{equation}\label{R}
|1-\langle Rz,\zeta\rangle|<\left(1-\frac{\gamma}{2}\right)(1-R)+\frac{\gamma}{2}(1-R|z|^2).
\end{equation}
If $1>|z|^2 > \alpha$ with $\alpha>(2/\gamma-1)$, then for $R$ close enough to $1$, we have $Rz \in \Gamma(\zeta)$. (Note that if $\gamma\geq 2$, any $z\neq 0$ and $R\in (0,1)$ will work.)\\

We will use a variation of a well-known argument, which can be found with details in \cite{Pel}, for instance. Define
$$\Gamma^h(\zeta)=\Gamma(\zeta)\setminus \overline{D(0,1/(h+1))},$$
and
$$A_\nu(g|h)(\zeta)=\left(\int_{\Gamma^h(\zeta)}|f(z)|^2d\nu(z)\right)^{1/2}.$$
Now set
$$h(\zeta)=\sup \big \{h:A_\nu(g|h)(\zeta) \leq C_1 C^*_\nu(g)(\zeta)\big \}.$$

The claim is proven, once we show
\[
\int_{\Bn} k(z)(1-|z|^2)^nd\nu(z)\leq 2\, \int_{\Sn}\left(\int_{\Gamma^{h(\zeta)}(\zeta)}k(z)d\nu(z)\right)d\sigma(\zeta)
\]
for every positive $\nu$-measurable function $k$.
By Fubini's theorem, the integral on the right-hand-side equals
\[
\int_{\Bn}\sigma(I(z)\cap H(z))\,k(z)\,d\nu(z),
\]
with $H(z)=\{\zeta \in \Sn: 1/(1+h(\zeta))\leq |z|\}$.  We therefore want to show that
\[
\frac{\sigma(I(z)\cap H(z))}{\sigma(I(z))}\geq \frac{1}{2}
\]
holds for $z\in M_K$. To this end, take $z \in M_K$ and set $z^*$ be the unique point in $\Bn$ satisfying $\zeta_{z^*}=\zeta_z$ and $(1-|z^*|^2)=K(1-|z|^2)$. Suppose, for now, that $K$ is chosen to be large enough so that the conditions on \eqref{R} are met. Then, there exists $r_K>0$ with the following properties.
\begin{itemize}
\item[(i)] If $0<r<r_K$, then $D(z^*,r)\subset \bigcap_{v \in I(z)} \Gamma(v)$;
\item[(ii)] If $z' \in D(z^*,r)$, then $2(1-|z'|^2)\geq (1-|z^*|^2)$;
\item[(iii)] If $z'\in D(z^*,r)$, and $z'_*$ is the unique point in $\Bn$ satisfying $\zeta_{z'}=\zeta_{z'_*}$ and $K(1-|z'_*|^2)=(1-|z'|^2)$, then $|1-\langle z,z'_*\rangle|\leq 2(1-|z|^2)$.
\end{itemize}
Now, suppose that $r>0$ above is the density of the lattice and set $z'$ to be a point of the lattice contained in $D(z^*,r)$. Notice that if $u$ satisfying $|u|\geq |z|$ (so that $1-|u|^2\leq 1-|z|^2$) does not belong to $Q(z')$, then by the property (ii) above (recall that $d(a,b)=|1-\langle a,b\rangle|^{1/2}$ satisfies the triangle inequality on $\overline{\Bn}$)
\begin{align*}
|1-\langle u,\zeta_z\rangle|^{1/2}&\geq |1-\langle u,\zeta_{z'}\rangle|^{1/2}-|1-\langle \zeta_z,\zeta_{z'}\rangle|^{1/2}\\
&\geq (K/2)^{1/2}(1-|z|^2)^{1/2}-|1-\langle \zeta_{z^*},\zeta_{z'}\rangle|^{1/2}.
\end{align*}
An application of Exercise 1.25 from \cite{ZhuBn} together with property (iii) above shows us that
$$|1-\langle \zeta_{z^*},\zeta_{z'}\rangle|\leq 8(1-|z|^2).$$
Let us now set $K$ to be a number big enough so that $((K/2)^{1/2}-8^{1/2})^2>K/3$. Then
 $u \notin Q(z')$ implies that $$|1-\langle u,\zeta_z\rangle|\geq (K/3)(1-|z|^2).$$
If now $\zeta \in I(z)$, then
$$|1-\langle u,\zeta\rangle|^{1/2}\geq |1-\langle u,\zeta_z\rangle|^{1/2}-|1-\langle z,\zeta_z\rangle|^{1/2}-|1-\langle z,\zeta\rangle|^{1/2},$$
so
$$|1-\langle u,\zeta\rangle|\geq \left((K/3)^{1/2}-1-(\gamma/2)^{1/2}\right)^2 (1-|z|^2)\geq (\gamma/2)(1-|u|^2),$$
when $K$ is large enough compared to $\gamma$. We now fix $K$ big enough to carry us through all the calculations above.

From this point onwards, we could follow the argument of \cite{Pel}, as the only real difference was that we had to choose the point $z'$ from the lattice. We will present the remaining details for the convenience of the reader.

By our choice of $K$ and $r$, if $|u|\geq |z|$ does not belong to $Q(z')$, then $I(u)\cap I(z)=\emptyset$. Thus, if $x=1/|z|-1$ (so that $\Gamma^x(\zeta)=\Gamma(\zeta)\setminus \overline{D(0,|z|)}$), then
\[
\begin{split}
\frac{1}{\sigma(I(z))}\int_{I(z)} &\left(\int_{\Gamma^x(\zeta)}|g(u)|^{2}d\nu(u)\right)d\sigma(\zeta)\\
&=\frac{1}{\sigma(I(z))}\int_{\{|z|<|u|<1\}}\sigma(I(z)\cap I(u))|g(u)|^2d\nu(u) \\
&\leq\frac{1}{\sigma(I(z))}\int_{Q(z')}\sigma(I(z)\cap I(u))|g(u)|^2d\nu(u)\\
&\leq\frac{C_3}{\sigma(I(z'))}\int_{Q(z')}(1-|u|^2)^n |g(u)|^2 d\nu(u)\\
&\leq C_3 \inf_{v \in I(z)}C^*_\nu(g)(v)^2.
\end{split}
\]
The last inequality holds, because by property (i), we have $z'\in D(z^*,r)\subset \Gamma(v)$ for all $v \in I(z)$.
Now, let us chooce $C_1$ so that $C_1^2>2C_3$. If $E(z)=\Sn\setminus H(z)$, then
\begin{align*}
\sigma(E(z)\cap I(z))&\leq \int_{I(z)}\frac{A_\nu(g|x)(\zeta)^2}{C_1^2 C_\nu^*(g)(\zeta)^2}d\sigma(\zeta)\\
&\leq \frac{1}{C_1^2 \inf_{v \in I(z)}C^*_\nu(g)(v)^2}\int_{I(z)}A_\nu(g|x)(\zeta)^2d\sigma(\zeta)\\
&<\sigma(I(z))/2.
\end{align*}
It follows that $\sigma(I(z)\cap H(z)) \geq \sigma(I(z))/2$ for $z \in M_K$.  Note also that the implicit constants in the estimate can be chosen to be independent of the measure $\nu$. This finishes the proof.
\end{proof}
\mbox{}
\\
We also need the following more general area function description of the Hardy spaces. The result is probably known to experts, but we were unable to find a proof in the literature.

\begin{proposition}\label{FracArea}
Let $f$ be holomorphic on $\Bn$, $0<p<\infty$ and $d\lambda_n(z)=dv_{-1-n}(z)$. If $s\geq -1$ and $t>0$, then the following are equivalent.
\begin{enumerate}
\item[(a)] $f \in H^p(\Bn)$;
\item[(b)] $\left(\int_{\Gamma(\zeta)}|Rf(z)|^2 (1-|z|^2)^2 d\lambda_n(z)\right)^{1/2} \in L^p(\Sn);$
\item[(c)] $\left(\int_{\Gamma(\zeta)}|R^{s,t}f(z)|^2 (1-|z|^2)^{2t} d\lambda_n(z)\right)^{1/2} \in L^p(\Sn).$
\end{enumerate}
Moreover, if $f(0)=0$, then the $L^p$ norms involved in all the items above are equivalent.
\end{proposition}

\begin{proof}
The equivalence of (a) and (b) can be found in \cite{P1}. So, we will prove that (b) and (c) are equivalent. To this end, we may clearly assume that $f(0)=0$.

Let us assume (b). Since $f(0)=0$, we have the estimate
$$f(w)=\int_{\Bn}Rf(u)L_\beta(w,u) dv_\beta(u),$$
where
$$L_\beta(w,u)=\int_0^1 \left(\frac{1}{(1-\langle w,\rho u\rangle)^{n+1+\beta}}-1\right)\frac{d\rho}{\rho}$$
valid for large enough $\beta$, see page 51 of \cite{ZhuBn}. Moreover, if $\beta=s+N$ for some positive integer $N$, we have by Proposition 5 of \cite{ZZ},
$$R^{s,t}\frac{1}{(1-\langle w,\rho u\rangle)^{n+1+\beta}}=\frac{\phi(\langle w,\rho u\rangle)}{(1-\langle w,\rho u\rangle)^{n+1+\beta+t}},$$
where $\phi$ is a one variable polynomial of degree $N$. Note that $R^{s,t}1=1$, so putting $\rho=0$ gives $1$ in the above identity. Now, it is straightforward to obtain
$$|R^{s,t}L_\beta (w,u)|\lesssim \frac{1}{|1-\langle w,u\rangle|^{n+\beta+t}}.$$
This allows us to obtain the bound
$$|R^{s,t}f(w)|\lesssim \int_{\Bn} \frac{|Rf(u)|}{|1-\langle w,u\rangle|^{n+\beta+t}}dv_\beta(u).$$
We may assume that $\beta=s+N>n\max(1,2/p)-n+1$. By standard estimates, this leads to
$$(1-|w|^2)^{2t}|R^{s,t}f(w)|^2 \lesssim \int_{\Bn} \frac{|Rf(u)|^2dv_{\beta}(u)}{|1-\langle w,u\rangle|^{n+\beta+t-1}}(1-|w|^2)^{t}.$$

Now, by Lemma \ref{FRgeneral}, if $\theta>n+\beta-1$, we have
\begin{align*}
&\int_{\Gamma(\zeta)}\left[(1-|w|^2)^t |R^{s,t}f(w)|\right]^2d\lambda_n(w)\\
\lesssim&\int_{\Bn}\left[(1-|w|^2)^t |R^{s,t}f(w)|\right]^2\left(\frac{1-|w|^2}{|1-\langle \zeta,w\rangle|}\right)^\theta d\lambda_n(w)\\
\lesssim&\int_{\Bn} \left(\frac{1-|w|^2}{|1-\langle \zeta,w\rangle|}\right)^\theta \left[ \int_{\Bn} \frac{|Rf(u)|^2dv_{\beta}(u)}{|1-\langle w,u\rangle|^{n+\beta+t-1}}\right](1-|w|^2)^{t} d\lambda_n(w)\\
\lesssim& \int_{\Bn} |Rf(u)|^2\left(\int_{\Bn}\frac{(1-|w|^2)^{\theta+t-n-1} dv(w)}{|1-\langle \zeta,w\rangle|^\theta |1-\langle w,u\rangle|^{n+\beta+t-1}}\right) dv_{\beta}(u)\\
\lesssim& \int_{\Bn} |Rf(u)|^2 (1-|u|^2)^2 \left(\frac{1-|u|^2}{|1-\langle \zeta,u\rangle|}\right)^{n+\beta-1} d\lambda_n(u)
\end{align*}
Now, since $n+\beta-1>n\max(1,2/p)$, we can use Lemma \ref{Gamma} to get
\begin{align*}
&\int_{\Sn} \left(\int_{\Gamma(\zeta)}\left[(1-|w|^2)^t |R^{s,t}f(w)|\right]^2d\lambda_n(w)\right)^{p/2} d\sigma(\zeta)\\
\lesssim &\int_{\Sn} \left(\int_{\Bn} |Rf(u)|^2 (1-|u|^2)^2 \left(\frac{1-|u|^2}{|1-\langle \zeta,u\rangle|}\right)^{n+\beta-1} d\lambda_n(u)\right)^{p/2}d\sigma(\zeta)\\
\lesssim &\int_{\Sn} \left(\int_{\Gamma(\zeta)} |Rf(u)|^2 (1-|u|^2)^2 d\lambda_n(u)\right)^{p/2}d\sigma(\zeta),
\end{align*}
so (c) is obtained.\\

Suppose now that (c) holds. By an estimate from  \cite[page 20]{ZZ}, for large enough $\beta$ we have
$$(1-|z|^2)|Rf(z)|\lesssim (1-|z|^2)\int_{\Bn}\frac{(1-|w|^2)^t|R^{s,t}f(w)|dv_\beta(w)}{|1-\langle z,w\rangle|^{n+2+\beta}}.$$
We may assume that $\beta>n\max(1,2/p)-n-1$ and let $0<\varepsilon<2$ so that $\beta-\varepsilon>-1$. Then use Cauchy-Schwarz and standard integral estimates to deduce
$$\left[(1-|z|^2)|Rf(z)|\right]^2 \lesssim (1-|z|^2)^2 \int_{\Bn}\frac{(1-|w|^2)^{2t}|R^{s,t}f(w)|^2 dv_{\beta+\varepsilon}(w)}{|1-\langle z,w\rangle|^{n+3+\beta}} (1-|z|^2)^{-\varepsilon}.$$
Now, take $\theta>n\max(1,2/p)+\varepsilon+1+\beta$, and estimate as before by using Lemma \ref{FRgeneral} to obtain
\begin{align*}
&\int_{\Gamma(\zeta)}\left[(1-|z|^2)|Rf(z)|\right]^2 d\lambda_n(z)\\
&\lesssim \int_{\Bn} \left[(1-|z|^2)|Rf(z)|\right]^2 \left(\frac{1-|z|^2}{|1-\langle z,\zeta\rangle|}\right)^{\theta} d\lambda_n(z)\\
&\lesssim \int_{\Bn} (1-|w|^2)^{2t}|R^{s,t}f(w)|^2 \left(\int_{\Bn} \frac{(1-|z|^2)^{2+\theta-\varepsilon-n-1}dv(z)}{|1-\langle z,\zeta\rangle|^\theta \,|1-\langle z,w\rangle|^{n+3+\beta}}\right)dv_{\beta+\varepsilon}(w)\\
&\lesssim \int_{\Bn} (1-|w|^2)^{2t}|R^{s,t}f(w)|^2 \left(\frac{1-|w|^2}{|1-\langle w,\zeta\rangle|}\right)^{n+1+\beta+\varepsilon} d\lambda_n(w).
\end{align*}
Since $n+1+\beta+\varepsilon>n\max(1,2/p)+\varepsilon$, from Lemma \ref{Gamma}  we conclude that (c) implies (b). This finishes the proof.
\end{proof}

\subsection{\textbf{Necessity}}
Throughout this proof, we set $\|Q_{\mu}\|=\|Q_{\mu}\|_{H^p\rightarrow H^q}$. Since $\mu(\Bn)=(Q_{\mu}1) (0)$, the measure $\mu$ is finite with $\mu(\Bn)\lesssim \|Q_{\mu}\|$. We split the proof in several cases.

\subsubsection{\textbf{The case $q=2$}}
It is well known that the boundedness is equivalent to
\[
\int_{\Bn}  | R^{-1,1}(Q_{\mu} f)(z)|^2 \,dv_ 1(z) \le C \,\|Q_{\mu}\|^2  \cdot \|f\|^2_{H^p}.
\]
For example, this can be deduced from Proposition \ref{FracArea} and the estimate \eqref{EqG}.\\

Let $\{a_ k\}\subset \Bn$ be a separated sequence and define
\[
f_ k(z)=\left (\frac{1-|a_ k|^2}{1-\langle z,a_ k \rangle}\right  )^{n+1}, \qquad z\in \Bn.
 \]
 Apply the previous inequality with the function $$F_ t(z)=\sum_ k \lambda_ k\,r_ k(t)\, f_ k(z)$$ with $\lambda=\{\lambda_ k\}\in T^p_ 2$, and use Proposition \ref{TKL} to get
\[
\int_{\Bn} \Big |\sum_ k \lambda_ k \,r_ k(t)\, R^{-1,1}(Q_{\mu} f_ k)(z)\Big |^2 \,dv_ 1(z) \le C \,\|Q_{\mu}\|^2 \cdot \|\lambda\|^2_{T^p_ 2}.
\]
Integrate respect to $t$ between $0$ and $1$, interchange the order of integration and use  Khinchine's inequality to obtain
\[
\sum_ k |\lambda_ k|^2\, \int_{\Bn} \big | R^{-1,1}(Q_{\mu} f_ k)(z)\big |^2 \,dv_ 1(z) \le C \,\|Q_{\mu}\|^2\cdot \|\lambda\|^2_{T^p_ 2}.
\]
By subharmonicity (see \cite[Lemma 2.24]{ZhuBn} for example) this implies
\[
\sum_ k |\lambda_ k|^2\,  \big | R^{-1,1}(Q_{\mu} f_ k)(a_ k)\big |^2 \,(1-|a_ k|)^{n+2} \le C \|Q_{\mu}\|^2 \cdot \|\lambda\|^2_{T^p_ 2}.
\]
and therefore
\[
\sum_ k |\lambda_ k|^2\,  \big | R^{-1,1}(Q_{\mu} f_ k)(a_ k)\big |^2 \,(1-|a_ k|)^{n+2} \le C \,\|Q_{\mu}\|^2\cdot \|\lambda^2\|_{T^{p/2}_ 1}.
\]
By the duality of tent spaces in Theorem \ref{TTD1}, this is equivalent to
\[
\left \{ \big | R^{-1,1}(Q_{\mu} f_ k)(a_ k)\big |^2 \,(1-|a_ k|)^2\right \} \in T_{\infty}^{p/(p-2)},
\]
with the corresponding $T_{\infty}^{p/(p-2)}$-norm dominated by $\|Q_{\mu}\|^2$. That is,
\[
\sup_{a_ k\in \Gamma (\zeta)} \big | R^{-1,1}(Q_{\mu} f_ k)(a_ k)\,(1-|a_ k|)\big |^2 \in L^{p/(p-2)}(\Sn)
\]
or
\begin{equation}\label{Eqq21}
\sup_{a_ k\in \Gamma (\zeta)} \big | R^{-1,1}(Q_{\mu} f_ k)(a_ k)\big |\,(1-|a_ k|) \in L^{2p/(p-2)}(\Sn)
\end{equation}
with the corresponding $L^{2p/(p-2)}$-norm dominated by $\|Q_{\mu}\|$.
However, since $(1-|a_ k|)\asymp |1-\langle a_ k,w \rangle |$ for $w\in Q(a_ k)$,
\[
\begin{split}
\big | R^{-1,1}(Q_{\mu} f_ k)(a_ k)| \,(1-|a_ k|)&= (1-|a_ k|^2)^{n+2} \left |\int_{\Bn} \!\!\frac{d\mu(w)}{|1-\langle a_ k,w \rangle |^{2(n+1)}} \right | \\
\\
& \ge (1-|a_ k|^2)^{n+2} \int_{Q(a_ k)} \frac{d\mu(w)}{|1-\langle a_ k,w \rangle |^{2(n+1)}} \\
\\
& \gtrsim \frac{1}{(1-|a_ k|)^n}  \int_{Q(a_ k)} (1-|w|)^n\frac{d\mu(w)}{(1-|w|)^n}.
\end{split}
\]

Let us restate this estimate for later reference.
\begin{equation}\label{estim}
\frac{\mu(Q(a_k))}{(1-|a_k|)^n} \le C\,\big |R^{-1,1}(Q_{\mu} f_ k)(a_ k)\big |\,(1-|a_ k|) .
\end{equation}
Finally, the result follows from \eqref{Eqq21}, \eqref{estim} and Lemma \ref{discrete} applied to the measure $d\nu(z)=d\mu(z)(1-|z|)^{-n}$ and $g=1$.

\subsubsection{\textbf{The case $q>2$}}
Let $\{a_k\}$ be a separated sequence in $\Bn$. Using the general area function description of $H^q$ given in Proposition \ref{FracArea}, and the same argument (applying Khinchine's inequality) and test functions as in the previous case, we arrive at
\begin{align*}
&\int_{\Sn} \left ( \sum_ k |\lambda_ k|^2 \int_{\Gamma (\zeta)} |R^{-1,1}(Q_{\mu} f_ k)(z)|^ 2\,dv_{1-n}(z)\right )^{q/2} d\sigma(\zeta)\\
\leq & C'\int_{\Sn}\left(\int_0^1 \int_{\Gamma(\zeta)} \left|\sum_k\lambda_k r_k(t)R^{-1,1}(Q_{\mu}f_k)(z)\right|^2dv_{1-n}(z)dt\right)^{q/2}d\sigma(\zeta)\\
\leq& C'\int_{0}^{1} \,\int_{\Sn}\left(\int_{\Gamma(\zeta)}\left|R^{-1,1}(Q_{\mu}F_ t)(z)\right|^2dv_{1-n}(z)\right)^{q/2}d\sigma(\zeta)\,dt \le C \,\|Q_{\mu}\|^q\cdot \|\lambda \|^q_{T^p_ 2}.
\end{align*}
Applying Lemma \ref{Gamma} with $\theta$ big enough we have
$$
\int_{\Sn} \left[\sum_{k}|\lambda_k|^2 \left(\frac{1-|a_k|^2}{|1-\langle a_k,\zeta\rangle|}\right)^\theta\int_{D_k}   |R^{-1,1}(Q_\mu f_k)(z)|^2 dv_{1-n}(z)\right]^{q/2}d\sigma(\zeta)\leq C\|Q_\mu\|^q\cdot \|\lambda \|^q_{T^p_ 2},
$$
where $D_ k$ denotes the Bergman metric ball $D(a_ k,r)$. Now, because $|1-\langle a_ k,\zeta \rangle | \asymp 1-|a_ k|$ for $a_ k\in \Gamma(\zeta)$, summing only over indices $k$ so that $a_k \in \Gamma(\zeta)$, we obtain
\[
\int_{\Sn} \left[\sum_{k:a_ k\in \Gamma (\zeta)}|\lambda_k|^2 \int_{D_k}   |R^{-1,1}(Q_\mu f_k)(z)|^2 dv_{1-n}(z)\right]^{q/2}d\sigma(\zeta)\leq C\|Q_\mu\|^q\cdot \|\lambda \|^q_{T^p_ 2}.
\]
By subharmonicity and the estimate \eqref{estim}, this gives
\begin{equation}\label{EqQ2b}
\int_{\Sn} \left ( \sum_ {a_ k \in \Gamma (\zeta)} |\lambda_ k|^2  \,\left (\frac{\mu(Q(a_ k))}{(1-|a_ k|)^n}\right )^2 \right )^{q/2} d\sigma(\zeta) \lesssim \|Q_\mu\|^q\|\cdot \|\lambda \|^q_{T^p_ 2}.
\end{equation}

Now, let $\beta=(pq-p-q)/(p-q)$ so that $q(\beta-1)/(q-2)=r$. Since $p>q>2$ we see that $\beta>1$.  By \eqref{EqG} we have
\begin{equation}\label{Eq-Q>2}
\begin{split}
\sum_ k |\lambda_ k|^2 &\left ( \frac{\mu(Q(a_ k))}{(1-|a_ k|)^n}\right )^{\beta+1} (1-|a_ k|)^n \asymp \int_{\Sn} \left (\sum_ { a_ k\in \Gamma(\zeta)} |\lambda_ k|^2 \left ( \frac{\mu(Q(a_ k))}{(1-|a_ k|)^n}\right )^{\beta+1}\right )d\sigma(\zeta)\\
\\
& \le \int_{\Sn} \left (\sup_{a_ k\in \Gamma(\zeta)} \frac{\mu(Q(a_ k))}{(1-|a_ k|)^n}\right )^{\beta-1}\left (\sum_ {k: a_ k\in \Gamma(\zeta)} |\lambda_ k|^2 \left (\frac{\mu(Q(a_ k))}{(1-|a_ k|)^n}\right )^2\right )d\sigma(\zeta).
\end{split}
\end{equation}

Hence, applying H\"{o}lder's inequality and \eqref{EqQ2b} we get
\[
\begin{split}
\sum_ k |\lambda_ k|^2 &\left ( \frac{\mu(Q(a_ k))}{(1-|a_ k|)^n}\right )^{\beta+1} (1-|a_ k|)^n   \lesssim \|\nu \|_{T^r_{\infty}}^{\beta-1} \cdot \big \|Q_{\mu}\big \|^2 \cdot \|\lambda \|^2_{T^p_ 2},
\end{split}
\]

with $\nu=\{\nu_ k\}$ and $\nu_  k =\mu(Q(a_ k))/(1-|a_ k|)^n$. That is, we have

\[
\sum_ k |\alpha_ k|\,\nu_ k ^{\beta+1} \le C  \|\nu \|_{T^r_{\infty}}^{\beta-1} \cdot \big \|Q_{\mu}\big \|^2 \cdot \|\alpha \|^2_{T^{p/2}_ 1}.
\]
By the duality between the tent spaces $T^{p/2}_ 1$ and $T_{\infty}^{(p/2)'}$ (see Theorem \ref{TTD1}), we obtain
\[
\big \|\nu^{\beta+1}\big \|_{T^{p/(p-2)}_{\infty}}\lesssim \|\nu \|_{T^r_{\infty}}^{\beta-1} \cdot \big \|Q_{\mu}\big \|^2 .
\]
It is straightforward to check that
\[
(\beta+1)p/(p-2)=r
\]
so that
\[
\big \|\nu \big \|^{\beta+1}_{T^r_{\infty}}\le \big \|\nu^{\beta+1}\big \|_{T^{p/(p-2)}_{\infty}}\lesssim \big \|\nu \big \|_{T^r_{\infty}}^{\beta-1} \cdot \big \|Q_{\mu}\big \|^2 .
\]
Now, if $\mu$ is compactly supported, then $\|\nu \big \|_{T^r_{\infty}}<\infty$, and therefore we deduce that
\[
\|\nu \big \|_{T^r_{\infty}} \lesssim \big \|Q_{\mu}\big \|
\]
An application of Lemma \ref{discrete} gives $\|\widetilde{\mu}\|_{L^r}\lesssim \|Q_{\mu}\|$ when $\mu$ is compactly supported. The general case follows by an approximation argument. Indeed, let $r_ k \in (0,1)$ with $r_ k\rightarrow 1$, and consider the measures $\mu_ k$ and $\mu_ k^*$ defined on Borel sets $E$ by $\mu_ k(E)= \mu(E\cap D(0,r_ k))$ and $\mu_ k^* (E)=\mu \big (r_ k^{-1}E \cap D(0,r_ k)\big )$. Note that $\widetilde{\mu_ k}\le \widetilde{\mu_ k^*}$ because $$D(0,r_ k)\cap \Gamma (\zeta) \subset D(0,r_ k)\cap r_ k^{-1}\Gamma (\zeta).$$ To be honest, if the aperture of $\Gamma(\zeta)=\Gamma_\gamma(\zeta)$ satisfies $\gamma\geq 2$, then the above inclusion holds as stated, whereas if $1<\gamma<2$, then it will hold modulo a compact set, see the beginning of the proof of Lemma \ref{discrete}. The argument can be completed either way, and we can always just change the aperture. We arrive at
\[
\big \|\widetilde{\mu}\big \|_{L^r} \le \liminf _ k\big \|\widetilde{\mu_ k}\big \|_{L^r} \le \liminf_ k \big \|\widetilde{\mu^*_ k}\big \|_{L^r}
\lesssim \liminf _ k \big \| Q_{\mu^*_ k}\big \|.
\]
Now, if $f$ is a unit vector in $H^p$, we have (by a change of variables) $|Q_{\mu^*_ k}f(z)| \asymp |(Q_{\mu} f_ k)_ k (z)|$, where $g_ k(z)=g(r_ k z)$, which easily gives $\|Q_{\mu_ k^*}\|\lesssim \|Q_{\mu}\|$ because dilatation by $r_ k$ can only decrease the norm. This gives
\[
\big \|\widetilde{\mu}\big \|_{L^r} \lesssim \big \| Q_{\mu}\big \|
\]
finishing the proof in this case.
\subsubsection{\textbf{The case $q<2$}} We will start with the following simple lemma.

\begin{lemma}\label{L-Qmu1}
Let $1<p<\infty$. If $Q_{\mu} f$ is in $H^p$ and  $g\in H^{p'}$, then
\begin{equation*}
\int_{\Sn} Q_\mu f(\zeta)\overline{g(\zeta)}d\sigma(\zeta)=\int_{\Bn} \!\!f(w)\,\overline{g(w)}\,d\mu(w).
\end{equation*}
\end{lemma}

\begin{proof}
If $g$ is a holomorphic polynomial, then
$$\Lambda_ g(h)=\int_{\Sn}h(\zeta)\overline{g(\zeta)}d\sigma(\zeta)$$ defines a bounded linear functional on $H^p$. Let $K_ w$ be the reproducing kernel of $H^2$ at the point $w$. Since $Q_{\mu} f\in H^p$, we have
\[
\begin{split}
\int_{\Sn} Q_\mu f(\zeta)\overline{g(\zeta)}d\sigma(\zeta)&=\Lambda_ g(Q_{\mu} f)=\Lambda_ g \left ( \int_{\Bn}\!\frac{f(w)}{(1-\langle \cdot ,w \rangle )^n}\,d\mu(w)  \right )\\
\\
&= \int_{\Bn} f(w) \,\Lambda_ g \left ( \frac{1}{(1-\langle \cdot ,w \rangle )^n}\right )\,d\mu(w)=\int_{\Bn} f(w) \,\Lambda_ g (K_ w)\,d\mu(w)\\
\\
&=\int_{\Bn} f(w) \,\langle K_ w,g \rangle _{H^2}\,d\mu(w)=\int_{\Bn} \!\!f(w)\,\overline{g(w)}\,d\mu(w).
\end{split}
\]
As the holomorphic polynomials are dense on $H^{p'}$, the result follows.
\end{proof}

It follows from Lemma \ref{L-Qmu1} that if $Q_\mu:H^p\to H^q$ is bounded, then its adjoint is $Q_\mu:H^{q'}\to H^{p'}$. If $p\le 2$, then  $p'\ge 2$ and from the previous cases (as $\frac{p'q'}{q'-p'}=\frac{pq}{p-q}=r$) we obtain that $\widetilde{\mu}$ is in $L^{r}(\Sn)$.

Hence we may assume that $1<q<2<p$. Let $\{a_k\}$ be a separated sequence in $\Bn$. Using the area function description of Hardy spaces obtained in Proposition \ref{FracArea}, Fubini's theorem, and the same test functions as before, we obtain

$$\int_{\Sn}\int_0^1 \left(\int_{\Gamma(\zeta)}\left|\sum_{k} \lambda_k r_k(t)R^{-1,1}(Q_\mu f_k)(z)\right|^2dv_{1-n}(z)\right)^{q/2}dt d\sigma(\zeta) \le C \|Q_\mu\|^q\|\lambda \|^q_{T^p_ 2},$$
where $r_k$ are the Rademacher functions. Proceeding with Kahane's inequality, we get
$$\int_{\Sn}\left(\int_0^1 \int_{\Gamma(\zeta)}\left|\sum_{k} \lambda_k r_k(t)R^{-1,1}(Q_\mu f_k)(z)\right|^2dv_{1-n}(z)dt\right)^{q/2} d\sigma(\zeta) \le C \|Q_\mu\|^q\|\lambda \|^q_{T^p_ 2}.$$

Next, Fubini's theorem together with Khinchine's inequality leads to

$$\int_{\Sn}\left( \int_{\Gamma(\zeta)}\sum_{k} |\lambda_k|^2 |R^{-1,1}(Q_\mu f_k)(z)|^2 dv_{1-n}(z)\right)^{q/2} d\sigma(\zeta) \le C \|Q_\mu\|^q\|\lambda \|^q_{T^p_ 2}.$$

Now, applying Lemma \ref{Gamma}, this gives for $\beta>2n/q$:

$$\int_{\Sn} \left[\sum_{k}|\lambda_k|^2 \int_{\Bn}  \left(\frac{1-|z|^2}{|1-\langle z,\zeta\rangle|}\right)^\beta |R^{-1,1}(Q_\mu f_k)(z)|^2 dv_{1-n}(z)\right]^{q/2}d\sigma(\zeta)\leq C\|Q_\mu\|^q\|\lambda \|^q_{T^p_ 2},$$
which implies
$$\int_{\Sn} \left[\sum_{k}|\lambda_k|^2 \left(\frac{1-|a_k|^2}{|1-\langle a_k,\zeta\rangle|}\right)^\beta\int_{D_k}   |R^{-1,1}(Q_\mu f_k)(z)|^2 dv_{1-n}(z)\right]^{q/2}d\sigma(\zeta)\leq C\|Q_\mu\|^q\|\lambda \|^q_{T^p_ 2},$$
where $D_ k$ denotes the Bergman metric ball $D(a_ k,r)$. Now, because $|1-\langle a_ k,\zeta \rangle | \asymp 1-|a_ k|$ for $a_ k\in \Gamma(\zeta)$, summing only over indices $k$ so that $a_k \in \Gamma(\zeta)$, we arrive at

\[
\int_{\Sn} \left ( \sum_ {a_ k \in \Gamma (\zeta)} |\lambda_ k|^2 \int_{D_k } |R^{-1,1}(Q_{\mu} f_ k)(z)|^ 2\,dv_{1-n}(z)\right )^{q/2} d\sigma(\zeta) \le C \|Q_\mu\|^q\|\lambda \|^q_{T^p_ 2}.
\]
By subharmonicity, we get
\begin{equation}\label{compare}
\int_{\Sn} \left ( \sum_ {a_ k \in \Gamma (\zeta)} |\lambda_ k|^2  |R^{-1,1}(Q_{\mu} f_ k)(a_ k)|^ 2\,(1-|a_ k|^2)^2\right )^{q/2} d\sigma(\zeta) \le C \|Q_\mu\|^q\|\lambda \|^q_{T^p_ 2}.
\end{equation}

Let us now denote $\nu=(\nu_k)$, where $\nu_k=\mu(Q(a_k))(1-|a_k|^2)^{-n}$ and set
$$s=\frac{p-q}{q(p-2)}.$$
By our choices $p>2>q$, it follows that $2s>qs>1$. Using \eqref{EqG} and two H\"olders (first with $2s$ and $(2s)'$, and then with $qs$ and $(qs)'$), we obtain
\begin{align*}
\sum_k |\lambda_k|^2 \nu_k^{1/s}(1-|a_k|)^n &\asymp \int_{\Sn} \left[ \sum_{a_k \in \Gamma(\zeta)} |\lambda_k|^2 \nu_k^{1/s}\right]d\sigma(\zeta)\\
&\leq \int_{\Sn} \left[ \sum_{a_k \in \Gamma(\zeta)} |\lambda_k|^2\right]^{(2s-1)/2s}\left[ \sum_{a_k \in \Gamma(\zeta)} |\lambda_k|^2 \nu_k^2 \right]^{1/2s}d\sigma(\zeta)\\
&\leq \|\lambda \|^{2-\frac{1}{s}}_{T^p_ 2}\cdot \left\lbrace\int_{\Sn}\left[ \sum_{a_k \in \Gamma(\zeta)}|\lambda_k|^2 \nu_k^2\right]^{q/2}d\sigma(\zeta)\right\rbrace^{1/qs}.
\end{align*}
Here we have used that
$$\frac{(2s-1)}{2s}\frac{qs}{(qs-1)}=\frac{p}{2}.$$

Since, by \eqref{estim}, we have  $\nu_k \lesssim |R^{-1,1}(Q_{\mu} f_ k)(a_ k)|\,(1-|a_ k|)$, we can combine this estimate with \eqref{compare}, and arrive at
$$\sum_k |\lambda_k|^2 \nu_k^{1/s}(1-|a_k|)^n \leq C\|Q_\mu\|^{1/s}\|\lambda^2 \|_{T^{p/2}_1}.$$

So, by the duality of tent spaces, we obtain $\|\nu^{1/s}\|_{T^{p/(p-2)}_\infty} \leq C\|Q_\mu\|^{1/s}$. Note that $(1/s)p/(p-2)=r$, so we get
$$\|\nu\|_{T^r_\infty}\leq C \|Q_\mu\|.$$

An application of Lemma \ref{discrete} finally shows that $\widetilde{\mu} \in L^r(\Sn)$ with $\|\widetilde{\mu}\|_{L^r}\lesssim \|Q_{\mu}\|$ as claimed.

\section{Compactness}\label{sComp}
For $1<p,q<\infty$, a linear operator $T:H^p\rightarrow H^q$ is compact if $\|Tf_ n\|_{H^q}\rightarrow 0$ whenever $\{f_ n\}$ is a bounded sequence in $H^p$ converging to zero uniformly on compact subsets of $\Bn$.\\

Recall also that a finite Borel measure on $\Bn$ is called a vanishing $s$-Carleson measure if for every $\zeta \in \Sn$
$$\mu(B_{\delta}(\zeta))\delta^{-ns}\to 0$$ as $\delta \to 0$. Equivalently, one may require
that for each (some) $t>0$ one has
\begin{equation}\label{sCM-1}
\lim_{|a|\rightarrow 1^-}\int_{\Bn} \!\!\frac{(1-|a|^2)^t}{|1-\langle a,w \rangle |^{ns+t}} \,d\mu(w)=0.
\end{equation}
Now we are ready for the description of the compactness of the Toeplitz type operator $Q_{\mu}$ acting between Hardy spaces.

\begin{theorem}\label{mtC-1}
Let $1<p\le q<\infty$ and $\mu$ be a positive Borel measure on $\Bn$. Then $Q_{\mu}: H^p\rightarrow H^q$ is compact if and only if $\mu$ is a vanishing $(1+\frac{1}{p}-\frac{1}{q})$-Carleson measure.
\end{theorem}
\begin{proof}
Assume first that $Q_{\mu}$ is compact. In the proof of the boundedness, we have seen that
\[
\int_{\Bn}\frac{(1-|z|^2)^{n/p'}d\mu(w)}{|1-\langle w,z \rangle |^{2n}} =|Q_{\mu}k_ z (z)| \le (1-|z|^2)^{-n/q}\big \|Q_{\mu}k_ z  \big \|_{H^q}.
\]
Assuming that $Q_ {\mu}$ compact, we know that $\|Q_{\mu}k_ z\|_{H^q}\rightarrow 0$ as $|z|\rightarrow 1^{-}$, and the result follows by \eqref{sCM-1}.

Conversely, suppose that $\mu$ is a vanishing $s$-Carleson measure with $s=1+1/p-1/q$. To prove the compactness of $Q_{\mu}$ we must show that $\|Q_{\mu} f_ n\|_{H^q}\rightarrow 0$ if $\{f_ n\}$ is a bounded sequence in $H^p$ that converges to zero uniformly on compact subsets of $\Bn$. As $\mu$ is an $s$-Carleson measure, by Theorem \ref{T2}, the Toeplitz operator  $Q_{\mu}:H^p\rightarrow H^q$ is bounded. Hence, by duality, Lemma \ref{L-Qmu1} and Carleson-Duren's theorem (argue as in the proof of Theorem \ref{T2}), we get
\[
\|Q_{\mu} f_ n\|_{H^q}=\sup_{\|g\|_{H^{q'}}=1} |\langle Q_{\mu} f_ n,g \rangle |\le \sup_{\|g\|_{H^{q'}}=1} \int_{\Bn} |f_ n(z)|\,|g(z)| \,d\mu(z)
\lesssim \|f_ n \|_{L^{ps}(\mu)}
\]
Since $\mu$ is a vanishing $s$-Carleson measure, then the embedding $i_{\mu}: H^p\rightarrow L^{ps} (\mu)$ is compact, and therefore $\|f_ n\|_{L^{ps}(\mu)} \rightarrow 0$ which proves that $\|Q_{\mu} f_ n\|_{H^q} \rightarrow 0$ finishing the proof.
\end{proof}

We also present the following, seemingly stronger version of Theorem \ref{T3}.

\begin{theorem}
Let $1<q<p<\infty$ and $\mu$ be a positive Borel measure on $\Bn$. Then $Q_{\mu}: H^p\rightarrow H^q$ is compact if and only if it is bounded, if and only if $\widetilde{\mu}\in L^r(\Sn)$, where $r=pq/(p-q)$.
\end{theorem}

\begin{proof} In view of Theorem \ref{T3}, it suffices to show that the condition  $\widetilde{\mu}\in L^r(\Sn)$ implies the compactness of $Q_{\mu}$.

For any compact $K\subset \Bn$, we set $\mu_K=\mu\chi_K$. Suppose that $\{f_n\}$ is a bounded sequence in $H^p$ converging to zero uniformly on compact subsets of $\Bn$. If $g$ is an arbitrary unit vector in $H^{q'}$, then by the standard pointwise estimate we have $|g(z)|\leq C_K$ on $K$ (uniformly on $g$). Hence, by duality and Lemma \ref{L-Qmu1},
\begin{align*}
\big \|Q_{\mu_ K} f_ n \big \|_{H^q}&=\sup_{\|g\|_{H^{q'}}=1}|\langle Q_{\mu_K}f_n,g\rangle|
\leq \sup_{\|g\|_{H^{q'}}=1} \int_K |f_n(z)g(z)|d\mu(z)\\
&\leq \,C_K\int_K|f_n(z)|d\mu(z)\to 0.
\end{align*}
So $Q_{\mu_K}$ is compact. Next, define $\mu_s=\mu\chi_{\overline{D(0,s)}}$, so that $Q_{\mu_s}$ is compact as just shown. Now,
$$\|\widetilde{\mu}-\widetilde{\mu_s}\|_{L^r}^r=\int_{\mathbb{S}^n} \left|\int_{\Gamma(\zeta)}\frac{(1-\chi_{\overline{D(0,s)}})(z)d\mu(z)}{(1-|z|^2)^n}\right|^rd\sigma(\zeta).$$
Write
$$\Phi_s(\zeta)=\left|\int_{\Gamma(\zeta)}\frac{\phi_s(z)d\mu(z)}{(1-|z|^2)^n}\right|^r,$$
where $\phi_s(z)=(1-\chi_{\overline{D(0,s)}})(z)$.
Because $\widetilde{\mu}\in L^r(\Sn)$, the function $\Phi_0$ is defined almost everywhere. Therefore, for almost every $\zeta \in \Sn$, the function
$$\frac{\phi_s(z)d\mu(z)}{(1-|z|^2)^n}$$ has an integrable majorant. Therefore, $\Phi_s(\zeta)\to 0$ as $s\to 1$ for almost every $\zeta \in \Sn$, by the Lebesgue dominated convergence theorem. Remembering that $\widetilde{\mu}\in L^r(\Sn)$, we see that $\Phi_s$ has also an integrable majorant, namely $\Phi_0$. Thus $\widetilde{\mu_s}\to \widetilde{\mu}$ in $L^r(\Sn)$, by another application of the dominated convergence theorem. By the norm estimate for $Q_\mu$ we have
\[
\big \|Q_{\mu}-Q_{\mu_ s} \big \| \lesssim \big \| \widetilde{\mu}-\widetilde{\mu_ s} \big \|_{L^r} \rightarrow 0
\]
as $s\rightarrow 1$, proving that $Q_{\mu}$ is compact.
\end{proof}

\section{Schatten class membership}\label{SSC}
In this section we are going to describe the membership of $Q_{\mu}$ in the Schatten ideals $S_ p(H^2)$. \\

Let $t>0$ and define
\[
S_ t \mu(w)=(1-|w|^2)^{n+t}\int_{\Bn}  \frac{d\mu(z)}{|1-\langle z,w \rangle |^{2n+t}}.
\]

We will denote by $d\lambda_n(z):=dv_{-1-n}(z)$ the invariant measure. For technical reasons, it will be convenient to denote

$$t_p=\max(n/p-n,0).
$$

\begin{proposition}\label{S-P1}
Let $\mu$ be a positive Borel measure on $\Bn$, and let $0<p<\infty$. The following conditions are equivalent:
\begin{enumerate}
\item[(a)] $S_ t\mu\in L^p(\Bn,d\lambda_ n)$ for all (some) $t>t_p$.

\item[(b)] The sequence $\Big\{\frac{\mu(D(a_ k,r))}{(1-|a_ k|^2)^n}\Big\}$ is in $\ell^p$, for any $r$-lattice $\{a_ k\}$.
\end{enumerate}
\end{proposition}

\begin{proof}
Denote by $D_k=D(a_k,r)$. By the properties of the lattice we have
\[
\begin{split}
\|S_ t \mu \|^p_{L^p(\lambda_ n)}&=\int_{\Bn} \left ( (1-|z|^2)^{n+t}\int_{\Bn} \frac{d\mu(w)}{|1-\langle z,w \rangle |^{2n+t}}\right )^p d\lambda_ n(z)\\
& \gtrsim \sum_ k \int_{D_ k} \left ( (1-|z|^2)^{n+t}\int_{D_ k} \frac{d\mu(w)}{|1-\langle z,w \rangle |^{2n+t}}\right )^p d\lambda_ n(z)\\
& \asymp \sum_ k \left (\frac{\mu(D_ k)}{(1-|a_ k|^2 )^{n}}\right )^p,
\end{split}
\]
because $|1-\langle z,w \rangle | \asymp (1-|z|^2)\asymp (1-|a_ k|^2)$ for $z,w\in D_ k$ (see \cite[Lemma 2.20]{ZhuBn}). Thus (a) implies (b) when $t>0$.

Assume that (b) holds. As the sets $D_ k$ cover $\Bn$ and $|1-\langle z,w \rangle |\asymp |1-\langle a_ k,w \rangle |$ for $z\in D_ k$ (by the estimate (2.20) in p.63 of \cite{ZhuBn}), we have
\[
S_ t \mu (w) \le (1-|w|^2)^{n+t} \sum_ k \int_{D_ k}\frac{d\mu(z)}{|1-\langle z,w \rangle |^{2n+t}}\lesssim (1-|w|^2)^{n+t}\sum_ k \frac{\mu(D_ k)}{|1-\langle a_ k,w \rangle|^{2n+t}}.
\]
This gives
\begin{equation}\label{EqP1}
S_ t \mu(w)^p \lesssim (1-|w|^2)^{np+tp}\left (\sum_ k \frac{\mu(D_ k)}{|1-\langle a_ k,w \rangle|^{2n+t}}\right )^{p}.
\end{equation}

For $p>1$, we take $0<\varepsilon<n\min (\frac{1}{p},\frac{1}{p'})$, and use H\"older's inequality and Lemma \ref{l2} to get
\begin{displaymath}
\begin{split}
S_ t \mu(w)^p
&\lesssim (1-|w|^2)^{(n+t)p} \left (\sum_ k \frac{\mu(D_ k)^p\,(1-|a_ k|^2)^{-n(p-1)-\varepsilon p}}{|1-\langle a_ k,w \rangle|^{2n+tp}}\right )
\left (\sum_ k \frac{(1-|a_ k|^2)^{n+\varepsilon p'}}{|1-\langle a_ k,w \rangle|^{2n}}\right )^{p/p'}
\\
&\lesssim (1-|w|^2)^{\varepsilon p+tp+n} \left (\sum_ k \frac{\mu(D_ k)^p\,(1-|a_ k|^2)^{-n(p-1)-\varepsilon p}}{|1-\langle a_ k,w \rangle|^{2n+tp}}\right ).
\end{split}
\end{displaymath}
Therefore, by the integral type estimate in Lemma \ref{IctBn},
\begin{displaymath}
\begin{split}
\int_{\Bn} S_ t\mu(w)^p \,d\lambda_ n (w)&\lesssim \sum_ k \left (\frac{\mu(D_ k)}{(1-|a_ k|^2)^n} \right )^p (1-|a_ k|^2)^{n-\varepsilon p}
\int_{\Bn} \frac{(1-|w|^2)^{(\varepsilon+t)p-1} dv(w)}{|1-\langle a_ k,w \rangle|^{2n+tp}}
\\
&\lesssim \sum_ k \left (\frac{\mu(D_ k)}{(1-|a_ k|^2)^n} \right )^p,
\end{split}
\end{displaymath}
and we get the result in this case.\\
\mbox{}
\\
If $0<p\le 1$, starting from \eqref{EqP1} we get
\[
S_ t \mu(w)^p \lesssim (1-|w|^2)^{np+tp}\sum_ k \frac{\mu(D_ k)^p}{|1-\langle a_ k,w \rangle|^{2np+tp}}.
\]
Then
\begin{displaymath}
\begin{split}
\int_{\Bn}  S_ t \mu(w)^p\,d\lambda_ n(w)
\lesssim \sum_ k \mu(D_ k)^p \int_{\Bn} \frac{(1-|w|^2)^{np+tp}}{|1-\langle w,a_ k \rangle |^{2np+tp}}\,d\lambda_ n(w)
\end{split}
\end{displaymath}
and the result follows from Lemma \ref{IctBn} because $t>t_ p$.
\end{proof}

Now we state the main result of this section, that characterizes the membership in $S_ p(H^2)$ of the Toeplitz type operator $Q_{\mu}$.

\begin{theorem}\label{t-SC}
Let $0<p<\infty$ and $\mu$ be a positive Borel measure on $\Bn$. The following are equivalent:
\begin{enumerate}
\item[(a)]  $Q_{\mu}$ belongs to $S_ p(H^2)$;

\item[(b)] $S_ t\mu \in L^p(\Bn,d\lambda_ n)$ for all (some) $t>t_p$;
\item[(c)] for any $r$-lattice $\{a_ k\}$, we have
$$\sum _ k \left ( \frac{\mu(D(a_ k,r))}{(1-|a_ k|^2)^n}\right )^p <\infty.$$
\end{enumerate}
\end{theorem}
From Proposition \ref{S-P1} we already know that (b) and (c) are equivalent. In order to obtain the equivalence with condition (a) we need to introduce some concepts and notation.\\

Recall that $H^2$ is a reproducing kernel Hilbert space with the reproducing kernel function given by
\[K_ z(w)=\frac{1}{(1-\langle w,z\rangle )^{n}},\quad z,w\in \Bn\]
with norm $\|K_ z\|_{H^2}=\sqrt{K_ z(z)}=(1-|z|^2)^{-n/2}$.   The normalized kernel functions are denoted by $k_ z=K_ z/\|K_ z\|_{H^2}$.  We also need to introduce some ``derivatives" of the kernel functions. For $z,w\in \Bn$ and $t>0$, define
\[K_ z^{t}(w)=\overline{R^{-1,t}K_ w (z)}=\frac{1}{(1-\langle w,z\rangle )^{n+t}}\]
and let $k^{t}_ z$ denote its normalization, that is, $k^{t}_ z=K^{t}_ z/\|K^{t}_ z\|_{H^2}$.  In particular, since $f
(z)=\langle f,K_ z\rangle_{H^2}$ whenever $f\in H^2$, one has
\begin{equation}\label{eq1}
 R^{-1,t}f (z)=\langle f,K^{t}_ z\rangle_{H^2},\qquad f\in H^2(\Bn).
\end{equation}
\mbox{}
\\
For $\alpha>-1$, the Bergman space $A^2_{\alpha}$ consists of those holomorphic functions $f$ in $\Bn$ with $\|f\|_{A^2_{\alpha}}^2 =\int_{\Bn} |f(z)|^2 dv_{\alpha}(z)<\infty$. It is well known that, for $f\in H^2$ with $f(0)=0$ and $s\geq -1$, one has
\[
\|f\|_{H^2}\asymp \|Rf\|_{A^2_ 1}\asymp \|R^{s,1}f\|_{A^2_ 1}.
\]
\mbox{}
\\
We need first the following lemma that can be found in \cite{P1}.

\begin{lemma}\label{KLS}
Let $T:H^2(\Bn)\rightarrow H^2(\Bn)$ be a positive operator. For $t> 0$ set
\[\widetilde{T^t}(z)=\langle Tk^{t}_ z,k^{t}_ z\rangle_{H^2},\quad z\in \Bn.\]
\begin{enumerate}
\item[(a)]  Let $0<p\le 1$. If $\widetilde{T^t} \in L^p(\Bn,d\lambda_ n)$ then $T$ is in $S_ p(H^2)$.

\item[(b)] Let $p\ge 1$. If $T$ is in $S_ p(H^2)$ then $\widetilde{T^t} \in L^p(\Bn,d\lambda_ n)$.
\end{enumerate}
\end{lemma}
\mbox{}
\\
\begin{lemma}\label{p-bound}
Let $0<p<\infty$, $\mu$ be a positive Borel measure on $\Bn$, and suppose that $S_ t\mu \in L^p(\Bn,d\lambda_ n)$ for $t>t_ p$. Then $Q_{\mu}$ is bounded on $H^2$.
\end{lemma}

\begin{proof}
By Theorem \ref{T2} it is enough to show that $\mu$ is a Carleson measure. Let $\{a_ k\}$ be an $r$-lattice on $\Bn$ and set $D_ k=D(a_ k,r)$. By \eqref{sCM} we need to prove that $\sup_{a\in \Bn} I_{\mu}(a)<\infty$, where
\[
I_{\mu}(a):=\int_{\Bn} \left (\frac{1-|a|^2}{|1-\langle a,z \rangle |^2}\right )^n \,d\mu(z).
\]
Since $|1-\langle a,z \rangle |\asymp |1-\langle a,a_ k\rangle |$ for $z\in D_ k$, we get
\[
I_{\mu}(a) \le (1-|a|^2)^n \sum_ k \int_{D_ k} \frac{d\mu(z)}{|1-\langle a,z \rangle |^{2n}}\lesssim (1-|a|^2)^n \sum _ k \frac{\mu(D_ k)}{|1-\langle a_ k, a\rangle |^{2n}}.
\]
If $0<p\le 1$, taking into account Proposition \ref{S-P1}, the result follows directly  from our assumption and the fact that $|1-\langle a_k,a\rangle|^2\geq (1-|a_k|)(1-|a|)$. If $p>1$, we use H\"{o}lder's inequality to get
\[
I_{\mu}(a) \lesssim (1-|a|^2)^n \left (\sum _ k \frac{(1-|a_ k|^2)^{np'}}{|1-\langle a_ k,a \rangle |^{2np'}}\right )^{1/p'}
\left (\sum_ k \Big (\frac{\mu(D_ k) }{(1-|a_ k|^2)^n}\Big )^p \right )^{1/p},
\]
and now the result is a consequence of the assumption, Proposition \ref{S-P1} and Lemma \ref{l2}.
\end{proof}

As a consequence of the previous lemmas, we get the following.
\begin{proposition}\label{PS1}
Let $0<p<\infty$, $t>t_ p$, and $\mu$ be a positive Borel measure on $\Bn$.
If $0 < p \leq 1$ and $S_{t}\mu$ is in $L^p(\Bn,d\lambda_{n}),$ then  $Q_\mu$ belongs to $S_p(H^2).$ Conversely, if $p\ge 1$ and $Q_\mu$ is in $S_p(H^2)$, then $S_{t}\mu \in L^p(\Bn,d\lambda_{n}).$
\end{proposition}
\begin{proof}
By Lemma  \ref{p-bound}, the condition $S_ t\mu \in L^p(\Bn,d\lambda_ n)$ implies the boundedness of $Q_{\mu}$  on $H^2$. Hence, in both cases, we have
\[
\langle Q_{\mu} k_ z^{t/2},k_ z^{t/2}\rangle=\int_{\Bn} |k^{t/2}_ z(w)|^2\,d\mu(w).
\]
As
\[
\|K_ z^{t/2} \|^2_{H^2}\asymp (1-|z|^2)^{-n-t},
\]
we obtain
\[
\langle Q_{\mu} k_ z^{t/2},k_ z^{t/2}\rangle \asymp S_{t} \mu(z).
\]
Therefore, the result is a direct consequence of Lemma \ref{KLS}.
\end{proof}

The next result, together with Proposition \ref{PS1}, gives that condition (b) implies (a) in Theorem \ref{t-SC}.

\begin{proposition}\label{P-SC2}
Let $p>1$, $\mu$ be a positive Borel measure on $\Bn$, and suppose that $S_ t \mu \in L^p(\Bn,d\lambda_ n)$ for $t>t_ p$. Then $Q_{\mu}$ belongs to $S_ p(H^2)$.
\end{proposition}

\begin{proof}
It is easy to modify the proof of Lemma \ref{p-bound} in order to see that the condition implies that $\mu$ is a vanishing Carleson measure. Hence, by Theorem \ref{mtC-1}, the operator $Q_{\mu}$ is compact. Thus, in view of  \cite[Theorem 1.27]{Zhu}, it is enough to show that
$$\sum_ k |\langle Q_{\mu} e_ k,e_ k \rangle |^p \le C<\infty$$
for all orthonormal sets $\{e_ k\}$ of $H^2$. In order to prove that, we
follow closely the argument used in \cite{PP}. By Lemma \ref{L-Qmu1} we have
\begin{displaymath}
|\langle Q_{\mu} e_ k,e_ k \rangle |^p =\left (\int_{\Bn} |e_ k(z)|^2 d\mu(z) \right )^p.
\end{displaymath}
By perturbing $Q_\mu$ by a rank-one operator, we may assume that $e_k(0)=0$ for all $k$. Since $e_ k^2 \in H^1 \subset A^2_{n-1}\subset A^1_{n-1}$, we have (see \cite[page 51]{ZhuBn})
\[
|e_ k(z)|^2 \lesssim \int_{\Bn} \frac{|R(e_ k^2)(w)|}{|1-\langle z,w \rangle |^{2n+t}}\,dv_{n+ t}(w)\lesssim \int_{\Bn} \frac{|e_ k(w)|\,|Re_ k(w)|}{|1-\langle z,w \rangle |^{2n+t}}\,dv_{n+ t}(w).
\]
This gives
\begin{equation}\label{Eq-Sc1}
|\langle Q_{\mu} e_ k,e_ k \rangle | \lesssim \int_{\Bn} |e_ k(w)|\,|Re_ k(w)| \,S_ t \mu(w)\,dv(w).
\end{equation}
If $1<p\le 2$, we use H\"{o}lder's inequality in \eqref{Eq-Sc1} and the fact that
\[
\|Re_ k\|_{A^2_ 1}\asymp \|e_ k\|_{H^2} =1
\]
to get
\[
\begin{split}
|\langle Q_{\mu} e_ k,e_ k \rangle |^p & \lesssim \ \int_{\Bn} |e_ k(w)|^p\,|Re_ k(w)|^{2-p} \,S_ t \mu(w)^p \,dv_{1-p}(w)
\end{split}
\]
Since
\[
\begin{split}
\sum_ k |e_ k(w)|^p\,|Re_ k(w)|^{2-p} & \le \Big (\sum_ k |e_ k(w)|^2 \Big )^{p/2}\Big (\sum_ k |Re_ k(w)|^2 \Big )^{(2-p)/2}
\\
& \le \|K_ w\|_ {H^2}^p\cdot \|RK_ w\|_{H^2}^{2-p}
\\
&\lesssim (1-|w|^2)^{-np/2} \,(1-|w|^2)^{-(n+2)(2-p)/2} =(1-|w|^2)^{p-(n+2)}
\end{split}
\]
we get
\[
\begin{split}
\sum_ k |\langle Q_{\mu} e_ k,e_ k \rangle | ^p &\lesssim \int_{\Bn} \left (\sum_ k |e_ k(w)|^p\,|Re_ k(w)|^{2-p} \right ) \,S_ t \mu(w)^p \,dv_{1-p}(w)
\\
&
\lesssim \int_{\Bn}  S_ t \mu(w)^p \,d\lambda_ n(w).
\end{split}
\]
This finishes the proof of the case $1<p\leq 2$.\\

If $p>2$, we use Cauchy-Schwarz inequality in \eqref{Eq-Sc1} to obtain
\[
\begin{split}
|\langle Q_{\mu} e_ k,e_ k \rangle |^2
&\lesssim \int_{\Bn} |e_ k(w)|^2\, S_ t\mu(w)^2\, (1-|w|^2)^n\,d\lambda_ n(w).
\end{split}
\]
By the reproducing formula for Bergman spaces, for any $\beta>0$, we have
\[
R^{\beta-1,1} e_ k(w)=\int_{\Bn} \frac{R^{\beta-1,1}e_ k(\zeta)\,dv_{\beta}(\zeta)}{(1-\langle w,\zeta\rangle)^{n+1+\beta}}.
\]
Hence,
\[
e_ k(w)=R_{\beta-1,1} R^{\beta-1,1} e_ k(w)=\int_{\Bn} \frac{R^{\beta-1,1}e_ k(\zeta)\,dv_{\beta}(\zeta)}{(1-\langle w,\zeta\rangle)^{n+\beta}}.
\]
Now, fix $\beta>n$ and take $\varepsilon>0$ with $\varepsilon <\min \big (\beta,\frac{n}{p-1}\big)$. By Cauchy-Schwarz and standard integral estimates we have
\[
\begin{split}
|e_ k(w)|^2 &\lesssim \left ( \int_{\Bn} \frac{|R^{\beta-1,1}e_ k(\zeta)|\,dv_{\beta}(\zeta)}{|1-\langle w,\zeta\rangle|^{n+\beta}} \right )^2
\\
&\lesssim \left ( \int_{\Bn} \frac{|R^{\beta-1,1}e_ k (\zeta)|^2\,dv_{1+\beta+\varepsilon}}{|1-\langle w,\zeta\rangle|^{n+\beta}}\right ) \,(1-|w|^2)^{-\varepsilon}.
\end{split}
\]
This gives
\[
|\langle Q_{\mu} e_ k,e_ k \rangle |^2 \lesssim \int_{\Bn} |R^{\beta-1,1}e_k(\zeta)| ^2\left ( \int_{\Bn} \frac{S_ t\mu(w)^2\, (1-|w|^2)^{n-\varepsilon}\,d\lambda_ n(w)}{|1-\langle w,\zeta\rangle|^{n+\beta}} \right ) \,dv_{1+\beta+\varepsilon}(\zeta).
\]
Because $\|R^{\beta-1,1}e_ k\|_{A^2_ 1}\asymp \|e_ k\|_{H^2}=1$, using H\"{o}lder's inequality with exponent $p/2>1$ we obtain
\[
|\langle Q_{\mu} e_ k,e_ k \rangle |^p \lesssim \int_{\Bn} |R^{\beta-1,1}e_ k(\zeta)|^2 \,C_{\mu}(\zeta)^{p/2} \,dv_{1+(\beta+\varepsilon)p/2} (\zeta),
\]
with
\[
C_{\mu}(\zeta):=\int_{\Bn} \frac{S_ t \mu(w)^2\,dv_{-1-\varepsilon}(w)}{|1-\langle w,\zeta \rangle|^{n+\beta}}.
\]
As
\[
\sum_ k |R^{\beta-1,1}e_ k(\zeta)|^2 \lesssim \big \| R^{\beta-1,1}K_{\zeta}\big \|_{H^2}\lesssim (1-|\zeta|^2)^{-(n+2)},
\]
we get
\[
\sum_ k |\langle Q_{\mu} e_ k,e_ k \rangle |^p \lesssim \int_{\Bn}  C_{\mu}(\zeta)^{p/2} \,(1-|\zeta|^2)^{(\beta+\varepsilon)p/2}\,d\lambda_ n (\zeta).
\]

Finally, by H\"{o}lder's inequality and the typical integral estimate,
\[
\begin{split}
C_{\mu}(\zeta)^{p/2}&\lesssim
\left (\int_{\Bn} \frac{S_ t \mu(w)^p}{|1-\langle w,\zeta \rangle |^{n+\beta}} \,dv_{-1-\varepsilon p+\varepsilon}(w) \right )
\left( \int_{\Bn} \frac{dv_{-1+\varepsilon}(w)}{|1-\langle w,\zeta \rangle |^{n+\beta}}
 \right )^{\frac{p}{2}-1}\\
 \\
&\lesssim \left (\int_{\Bn} \frac{S_ t \mu(w)^p}{|1-\langle w,\zeta \rangle |^{n+\beta}} \,dv_{-1-\varepsilon p+\varepsilon}(w) \right )
(1-|\zeta|^2)^{(\varepsilon-\beta) (\frac{p}{2}-1)}.
 \end{split}
\]
Inserting this into the previous estimate, using Fubini's theorem and Lemma \ref{IctBn} we conclude that
\[
\begin{split}
\sum_ k |\langle Q_{\mu} e_ k,e_ k \rangle |^p & \lesssim \int_{\Bn} S_ t \mu(w)^p \left (\int_{\Bn} \frac{(1-|\zeta|^2)^{\beta+\varepsilon (p-1)}}{|1-\langle w,\zeta \rangle|^{n+\beta}} \,d\lambda_ n(\zeta)\right )\, dv_{-1-\varepsilon p+\varepsilon}(w)
\\
& \lesssim \int_{\Bn} S_ t \mu(w)^p\,d\lambda_ n(w).
\end{split}
\]
This finishes the proof.
\end{proof}
\mbox{}
\\
\begin{proposition}\label{FPSC}
Let $0<p\le 1$, $\mu$ be a positive Borel measure on $\Bn$, and suppose that $Q_{\mu}\in S_ p(H^2)$. Then, for any $r$-lattice $\{a_ k\}$ on $\Bn$, we have
\[
\sum_ k \left ( \frac{\mu(D(a_ k,r))}{(1-|a_ k|^2)^n}\right )^p<\infty.
\]
\end{proposition}

\begin{proof}
 Fix $R>0$ big enough to be chosen later, and partition the lattice $\{a_ k\}$ into $M$ subsequences such that any two distinct points $b_ j$ and $b_ {\ell}$ in the same subsequence satisfy $\beta(b_ j,b_ {\ell})\ge R$. Let $\{b_ j \}$ be such a subsequence and consider the measure

\[
\nu=\sum_ j \mu \chi_ j,
\]
where $\chi_ j$ denotes the characteristic function of the Bergman metric ball $D_ j:=D(b_ j,r)$. Fix an orthonormal basis $\{e_ j\}$ of $H^2$, and consider the linear operator $A:H^2\rightarrow H^2$ defined by $Ae_ j=h_ j^t$ for a sufficiently large $t$, where
\[
h_ j^t(z)=\frac{(1-|b_ j|^2)^{(n+2t)/2}}{(1-\langle z,b_ j\rangle )^{n+t}}.
\]
For a large $t$, the operator $A$ is bounded on $H^2$. Then the operator $S=A^*Q_{\nu}A$ is in  $S_ p(H^2)$ with $\|S\|_{S_ p}\lesssim \|Q_{\nu}\|_{S_ p}$. Split the operator as $S=D+E$, where $D$ is a diagonal operator. We have
\[
\|S\|^p_{S_ p}\ge \|D\|^p_{S_ p}-\|E\|^p_{S_ p}.
\]
We can estimate the diagonal term as follows.
\[
\begin{split}
\|D\|^p_{S_ p}&=\sum_ j \big | \langle De_ j,e_ j\rangle \big |^p=\sum_ j \big | \langle Q_{\nu}h^t_ j,h^t_ j\rangle \big |^p\\
\\
& =\sum_ j \left ( \int_{\Bn} |h_ j^t (z)|^2\,d\nu(z) \right )^p \ge \sum_ j \left ( \int_{D_ j} |h_ j^t (z)|^2\,d\nu(z) \right )^p\\
\\
& \ge C \sum_ j \left ( \frac{\mu(D_ j)}{(1-|b_ j|^2)^n} \right )^p.
\end{split}
\]
For the non-diagonal term, we can proceed almost exactly as in the proof of Lemma 12 in \cite{ZhuNYJM}. We only sketch some details. First, one has the estimate

$$\|E\|_p^p \le C\sum_{i=1}^\infty \mu(D_i)I_i,$$
where
$$I_i=\sum_{j\neq m} |h^t_j(b_i)h^t_m(b_i)|^p=\sum_{j\neq m} \frac{[(1-|b_j|^2)(1-|b_m|^2)]^{pt+pn/2}}{[|1-\langle b_i,b_j\rangle||1-\langle b_i,b_m\rangle|]^{pn+pt}}.$$
Denote
$$G_R=\{(z,w) \in \Bn\times \Bn : \beta(z,w)\geq R-2r\}.$$
Then,
$$I_i\leq C (1-|b_i|^2)^{-pn}\int\int_{G_R} \frac{[(1-|z|^2)(1-|w|^2)]^{pt+pn/2-n-1}dv(z)dv(w)}{[|1-\langle z,b_i\rangle||1-\langle w,b_i\rangle|]^{pt}}.$$
Since $(1-|z|^2)(1-|w|^2)\leq 4|1-\langle z,b_i\rangle||1-\langle w,b_i\rangle|$, we obtain
$$I_i\leq C (1-|b_i|^2)^{-pn}\int\int_{G_R} \frac{dv(z)dv(w)}{[|1-\langle z,b_i\rangle||1-\langle w,b_i\rangle|]^{n+1-pn/2}}.$$
Now, choose $x\in (1,\infty)$ so that
$$A=x(n+1-pn/2)<n+1.$$
It follows that
$$I_i\leq C(1-|b_i|^2)^{-pn} v^2(G_R)^{1/x'},$$
where
$$v^2(G_R)=\int\int_{G_R}dv(z)dv(w)$$ tends to zero as $R\to \infty$.
Therefore, there exists a constant $C'$ so that
$$\|S\|_p^p \geq (C-C'v^2(G_R)^{1/x'})\sum_{i=1}^\infty \left(\frac{\mu(D_i)}{(1-|b_i|^2)^n}\right)^p.$$
Recall that the original lattice was partitioned into $M$ subsequences that all satisfy this bound. Therefore the claim follows.
\end{proof}
\subsection{Proof of Theorem \ref{t-SC}}
The equivalence between (b) and (c) is Proposition \ref{S-P1}. The implication (b) implies (a) is proved in Proposition \ref{PS1} (case $0<p\le1$) and Proposition \ref{P-SC2} (case $p\ge 1$). Finally, Proposition \ref{PS1} gives the implication (a) implies (b) for $p\ge 1$, and the implication (a) implies (c) for $0<p\le 1$ is Proposition \ref{FPSC}.

\section{Applications to weighted composition operators, Volterra type integration operators and Carleson embeddings}
In this section we provide applications of our results on the Toeplitz type operator $Q_{\mu}$ in order to  study the membership in the Schatten ideal $S_ p$ of weighted composition and Volterra type integral operators acting on $H^2$. We also obtain a description of when the Carleson embedding $R_{\mu}(f)=f$ is in $S_ p(H^2,L^2(\mu))$ for positive Borel measures $\mu$ in $\Bn$.
\\

For a holomorphic function $\varphi:\Bn \rightarrow \Bn$, the composition operator $C_{\varphi}$ is given by $C_{\varphi} f=f\circ \varphi$ for $f\in H(\Bn)$. Observe that, as $\varphi$ is a bounded holomorphic function on $\Bn$, it has radial limits $\varphi^*$ almost everywhere. Given a function $u\in H(\Bn)$, the weighted composition operator $W_{u,\varphi}$ is given by $W_{u,\varphi} f=uC_{\varphi} f$ for $f\in H(\Bn)$. We refer to \cite{CD1,CD2,Sty,ZC} for studies of weighted composition operators acting on Hardy spaces. If $W_{u,\varphi}$ is acting on $H^2$, then clearly $u\in H^2$, and therefore it has radial limits $u^*$ almost everywhere. Also it is well known and easy to see that $(W_{u,\varphi})^* W_{u,\varphi}=Q_{\mu_{u,\varphi}}$, where $\mu_{u,\varphi}$ is the measure defined on Borel sets $E\subset \Bn$ by
\[
\mu_{u,\varphi}(E)=\int_{ \varphi^{-1}(E)\cap \,\Sn} |u^*(\zeta)|^2\,d\sigma ( \zeta ).
\]
 Since $W_{u,\varphi}$ belongs to $S_ p(H^2)$ if and only if $Q_{\mu_{u,\varphi}}\in S_{p/2}(H^2)$, and
\[
\int_{\Bn}  \frac{d\mu_{u,\varphi}(z)}{|1-\langle z,w \rangle |^{2n+t}}=\int_{\Sn} \frac{|u^*(\zeta)|^2\,d\sigma(\zeta)}{|1-\langle \varphi^*(\zeta),w \rangle |^{2n+t}},
\]
in view of Theorem \ref{t-SC}, we obtain the following description of when the weighted composition operator $W_{u,\varphi}$ belongs to the Schatten ideal $S_ p(H^2)$.
\begin{theorem}\label{FinalT}
Let $0<p<\infty$, $u\in H(\Bn)$ and $\varphi:\Bn \rightarrow \Bn$ be holomorphic. Then $W_{u,\varphi}\in S_ p(H^2)$ if and only if, for $t>t_ {p/2}$ one has
\begin{equation}\label{FEq2}
\int_{\Bn}\left ((1-|w|^2)^{n+t}\int_{\Sn} \frac{|u^*(\zeta)|^2\,d\sigma(\zeta)}{|1-\langle \varphi^*(\zeta),w \rangle |^{2n+t}}\right )^{p/2}\,d\lambda_ n(w)<\infty.
\end{equation}
\end{theorem}
As a consequence we can recover the result of Luecking and Zhu \cite{LZ} on the membership in the Schatten classes of composition operators acting on the Hardy space of the unit disk $\mathbb{D}$. We use the Nevanlinna counting function $N^*_{\varphi}$ defined as
\[
N^*_{\varphi}(z)=\sum _{w: \varphi(w)=z} (1-|w|^2),\qquad z\in \mathbb{D}\setminus \{\varphi(0)\},
\]
where the last sum is interpreted as being zero if $z\notin \varphi(\mathbb{D})$.
\begin{coro}
Let $\varphi:\mathbb{D}\rightarrow \mathbb{D}$ analytic and $0<p<\infty$. Then $C_{\varphi}\in S_ p(H^2)$ if and only if
\[
\frac{N^*_{\varphi}(z)}{(1-|z|)} \in L^{p/2}(\mathbb{D},d\lambda_ 1).
\]
\end{coro}

\begin{proof}
For $w\in \mathbb{D}$ and $t>t_ {p/2}$, consider the function
$
f_ w(z)=(1-z\overline{w} )^{-1-t/2}
$, and observe that
\[
\int_{\mathbb{S}_1} \frac{d\sigma(\zeta)}{|1- \varphi^*(\zeta)\overline{w} |^{2+t}}=\|C_{\varphi}f_ w \|_{H^2}^2\asymp |C_{\varphi} f_ w(0)|^2 +\|(C_{\varphi} f_ w)'\|^2_{A^2_ 1}.
\]
By the change of variables formula (see \cite{CMc}),
\[
\|(C_{\varphi} f_ w)'\|^2_{A^2_ 1}=2 \int_{\D} |f'_ w(\varphi(z))|^2\,|\varphi'(z)|^2\,(1-|z|^2) \,dA(z)=2\int_{\D} |f'_ w(\zeta)|^2 N^*_{\varphi}(\zeta) dA(\zeta),
\]
where $dA$ is the normalized area measure on $\mathbb{D}$. In view of Theorem \ref{FinalT}, we have $C_{\varphi}\in S_ p(H^2)$ if and only if
\begin{equation}\label{FE4}
I(\varphi):=\int_{\D}\left ((1-|w|^2)^{1+t}\int_{\D}  \frac{N^*_{\varphi}(z)\,dA(z)}{|1- z\overline{w}|^{4+t}} \right )^{p/2}\,d\lambda_ 1(w)<\infty.
\end{equation}
From here, the necessity follows from the fact that $N_{\varphi}^*$ satisfies the inequality (a consequence of \cite[Lemma 3.18]{CMc})
\[
\frac{N_{\varphi}^*(w)}{(1-|w|^2)^{2+t}}\lesssim \int_{\D}  \frac{N^*_{\varphi}(z)\,dA(z)}{|1-z\overline{w} |^{4+t}}.
\]
When $p\ge 2$, from \eqref{FE4} we obtain the sufficiency after an application of H\"{o}lder's inequality with exponent $p/2$, and the typical integral estimate in Lemma \ref{IctBn}. The case $0<p<2$ requires more work. Take an $r$-lattice $\{a_ k\}$, and let $D_ k=D(a_ k,r)$. Then
\[
\int_{\D}  \frac{N^*_{\varphi}(z)\,dA(z)}{|1-z\overline{w} |^{4+t}}\lesssim \sum_ k \frac{1}{|1-a_ k\overline{w} |^{4+t}}\int_{D_ k}  N^*_{\varphi}(z)\,dA(z).
\]
By \cite[Proposition 11.4]{Zhu}, for $z\in D_ k$ we have
\[
N^*_{\varphi}(z)^{p/2} \lesssim \frac{1}{(1-|z|^2)^2}\int_{D(z,r)} N^*_{\varphi}(\xi)^{p/2}dA(\xi)\lesssim \int_{\widetilde{D}_ k} N^*_{\varphi}(\xi)^{p/2}d\lambda_ 1(\xi),
\]
with $\widetilde{D}_ k=D(a_ k,2r)$. This gives
\[
\int_{\D}  \frac{N^*_{\varphi}(z)\,dA(z)}{|1- z\overline{w} |^{4+t}}\lesssim \sum_ k \frac{(1-|a_ k|^2)^2}{|1- a_ k\overline{w} |^{4+t}}\left (\int_{\widetilde{D}_ k} N^*_{\varphi}(\xi)^{p/2}d\lambda_ 1(\xi)\right )^{2/p}.
\]
Since $0<p/2<1$, we get
\[
\left (\int_{\D}  \frac{N^*_{\varphi}(z)\,dA(z)}{|1-z\overline{w} |^{4+t}} \right )^{p/2} \lesssim \sum_ k \frac{(1-|a_ k|^2)^p}{|1- a_ k\overline{w} |^{(4+t)p/2}}\int_{\widetilde{D}_ k} N^*_{\varphi}(\xi)^{p/2}d\lambda_ 1(\xi).
\]
Inserting this into \eqref{FE4} and applying the typical integral estimate, we obtain
\[
\begin{split}
I(\varphi) &\lesssim \sum_ k (1-|a_ k|^2)^p \left (\int_{\mathbb{D}}\frac{(1-|w|^2)^{(1+t)p/2} d\lambda_ 1(w)}{|1-a_ k\overline{w} |^{(4+t)p/2}}\right)\int_{\widetilde{D}_ k} N^*_{\varphi}(\xi)^{p/2}d\lambda_ 1(\xi)
\\
& \lesssim \sum_ k (1-|a_ k|^2)^{-p/2} \int_{\widetilde{D}_ k} N^*_{\varphi}(\xi)^{p/2}d\lambda_ 1(\xi).
\end{split}
\]
As $(1-|a_ k|)\asymp (1-|\xi|)$ for $\xi \in \widetilde{D}_ k$, and because any point $\xi \in \mathbb{D}$ belongs to at most $N$ of the sets $\widetilde{D}_ k$, we finally get
\[
I(\varphi) \lesssim \sum_ k  \int_{\widetilde{D}_ k} \left(\frac{N^*_{\varphi}(\xi)}{1-|\xi|}\right )^{p/2}d\lambda_ 1(\xi) \lesssim \int_{\mathbb{D}} \left(\frac{N^*_{\varphi}(\xi)}{1-|\xi|}\right )^{p/2}d\lambda_ 1(\xi).
\]
This finishes the proof.
\end{proof}
\mbox{}
\\
For a function $g\in H(\Bn)$, the Volterra type integration operator $J_ g$ is defined as
\[
J_ g f(z)=\int_{0}^{1} f(tz) Rg(tz) \frac{dt}{t},\qquad z\in \Bn
\]
for $f$ holomorphic in $\Bn$. An easy calculation  provides the  basic formula $R(J_ g f)=fRg$ involving the radial derivative $R$ and the operator $J_ g$.
If $J_ g:H^2\rightarrow H^2$,  then it is well known (and easy to see) that $J_ g^* J_ g =Q_{\mu_ g} $ with $d\mu_ g=|Rg|^2dv_ 1$, where now $J_ g^*$ denotes the Hilbert adjoint respect to the inner product $\langle f,g \rangle_ *=f(0)\overline{g(0)}+\int_{\Bn} Rf\,\overline{Rg}\,dv_ 1$. Therefore, $J_ g$ is in $S_ p(H^2)$ if and only if $Q_{\mu_ g}$ belongs to $S_{p/2}(H^2)$. By Theorem \ref{t-SC}, this is equivalent to
\begin{equation}\label{FEq1}
\int_{\Bn}\left ((1-|w|^2)^{n+t}\int_{\Bn}  \frac{|Rg(z)|^2\,dv_ 1(z)}{|1-\langle z,w \rangle |^{2n+t}}\right )^{p/2}\,d\lambda_ n(w)<\infty,
\end{equation}
for $t>t_ {p/2}$. From this condition, it is easy to see that $J_ g \in S_ p(H^2)$ if and only if $g$ belongs to the analytic Besov space $B_ p$ when $p>n$, and $g$ being constant when $0<p\le n$. This recovers the results from \cite{AS1} and \cite{P1}. Indeed, if \eqref{FEq1} holds, by ``subharmonicity" we get
\[
\int_{\Bn} \left ( (1-|w|^2)\,|Rg(w)|\right )^{p}\,d\lambda_ n(w)<\infty.
\]
Hence \eqref{FEq1} implies that $g\in B_ p$ for $p>n$ and $g$ constant when $0<p\le n$. Finally, if $p>n$ and $g\in B_ p$ we see that condition \eqref{FEq1} is satisfied when $p\ge 2$ after an application of H\"{o}lder's inequality and the typical integral estimate.  When $n=1$ and $1<p<2$, we apply the inequality
\[
|Rg(z)|^p \lesssim \frac{1}{(1-|z|^2)^{n+1}} \int_{D(z,r)} |Rg(\xi)|^p \,dv(\xi)
\]
 and argue in a similar manner as in the case of composition operators.

The result obtained for the integration operator $J_ g$ can be generalized. By the Littlewood-Paley inequalities we have
\[
\|J_ g f \|_{H^2}^2 \asymp \|R(J_ g f)\|^2_{A^2_ 1} =\int_{\Bn} |f(z)|^2 \,d\mu_ g(z).
\]
Hence we can view the operator $J_ g$ on $H^2$ as the Carleson embedding from $H^2$ to $L^2(\Bn,d\mu_ g)$. Next, we are going to characterize, for a positive Borel measure $\mu$, when the Carleson embedding $R_{\mu}: H^2 \rightarrow L^2(\mu):=L^2(\Bn,d\mu)$ is in $S_ p$. An easy computation yields $R^*_{\mu}R_{\mu}=Q_{\mu}$, and from Theorem \ref{t-SC} we obtain the following description.
\begin{theorem}
Let $\mu$ be a positive Borel measure on $\Bn$ and $0<p<\infty$. Then the Carleson embedding $R_{\mu}$ is in $S_ p(H^2,L^2(\mu))$ if and only if $S_ t\mu \in L^{p/2}(\Bn,d\lambda_ n)$ for $t>t_{p/2}$.
\end{theorem}

This characterization can be compared with the result obtained by Smith \cite{Sm} in the case of the unit disk.
% ------------------------------------------------------------------------

\end{document}